\documentclass[a4paper,12pt]{amsart}
\usepackage[french, english]{babel}
\usepackage{amsfonts}
\usepackage{esint}
 \usepackage{color}
\usepackage{pifont}
\usepackage{latexsym,amsmath,amssymb,cite,amsthm}
\usepackage[colorlinks=true, linkcolor=blue,  anchorcolor=blue,citecolor=blue]{hyperref}
\usepackage{hyperref}
\usepackage[pagewise]{lineno}

\usepackage{color,eucal,enumerate,mathrsfs}
\usepackage[normalem]{ulem}
\usepackage{amsmath,amssymb,epsfig,bbm}

\numberwithin{equation}{section}
\theoremstyle{plain}
\newtheorem{theorem}{Theorem}[section]

\newtheorem{lemma}[theorem]{Lemma}
\newtheorem{proposition}[theorem]{Proposition}

\theoremstyle{definition}

\newtheorem{definition}[theorem]{Definition}
\theoremstyle{remark}
\newtheorem{assumption}[theorem]{\textbf{Assumption}}
\newtheorem{remark}[theorem]{Remark}

\newcommand{\lmt}[2]{\mathop{\lim}_{{#1} \rightarrow {#2}} }

\newcommand{\lip}[1]{{\mathrm{lip}}({#1})}

\newcommand{\lmts}[2]{\mathop{\overline{\lim}}_{{#1} \rightarrow {#2}} }
\newcommand{\lmti}[2]{\mathop{\underline{\lim}}_{{#1} \rightarrow {#2}} }

\newcommand{\mm}{\mathfrak m}
\newcommand{\ms}{(X,\d,\mm)}

\newcommand{\rcdkn}{{\rm RCD}(K, N)}
\newcommand{\rcd}{{\rm RCD}(K, \infty)}

\newcommand{\vol}{{\rm Vol}_{\rm g}}
\newcommand{\g}{{\rm g}}

\newcommand{\E}{\mathcal{E}}

\newcommand{\N}{\mathbb{N}}
\newcommand{\NN}{\mathcal{N}}

\newcommand{\R}{\mathbb{R}}
\newcommand{\C}{\mathfrak{C}}

\newcommand{\supp}{\mathop{\rm supp}\nolimits}   
\newcommand{\Lip}{\mathop{\rm Lip}\nolimits}

\renewcommand{\d}{{\mathrm d}}

\newcommand{\D}{{\mathrm D}}

\newcommand{\restr}[1]{\lower3pt\hbox{$|_{#1}$}}

\newcommand{\nchi}{{\raise.3ex\hbox{$\chi$}}}

\begin{document}

\title[Bourgain-Brezis-Mironescu's theorem revisited]{On the asymptotic behaviour of the  fractional Sobolev seminorms  in metric measure spaces:  Bourgain-Brezis-Mironescu's theorem revisited}


\author[B. ~Han]{Bang-Xian Han}
\address{Wu Wen-Tsun Key Laboratory of Mathematics and School of   Mathematical Sciences, University of Science and Technology of China (USTC), 230026, Hefei, China}
\email{hanbangxian@ustc.edu.cn}

\author[A.~Pinamonti]{Andrea Pinamonti}
\address{Dipartimento di Matematica\\ Università di Trento\\
Via Sommarive, 14, 38123 Povo TN}
\email{andrea.pinamonti@unitn.it}

\thanks{Part of this project was finished when the first author was a post-doctoral fellow at Technion-Israel Institute of Technology, supported by  the European Research Council (ERC) under the European Union's Horizon 2020 research and innovation programme (grant agreement No 637851)}

\date{\today}
\bibliographystyle{abbrv}
\maketitle

\begin{abstract}
We generalize Bourgain-Brezis-Mironescu's asymptotic  formula for fractional Sobolev functions,  in the setting of metric measure spaces,  under the assumption that at almost every point the tangent space in the  measured Gromov-Hausdorff sense is a finite dimensional Banach space or a Carnot group. Our result not only covers the known results concerning Euclidean spaces, weighted Riemannian manifolds and finite dimensional  Banach spaces,   but also  extends  Bourgain-Brezis-Mironescu's  formula to  $\rcdkn$ spaces and sub-Riemannian manifolds. 
\end{abstract}

\textbf{Keywords}: metric measure space, fractional Sobolev space, non-local functional,  measured Gromov-Hausdorff convergence, tangent cone.\\
\tableofcontents

\section{Introduction}\label{first section}
In the early 2000’s the study of fractional $s$-seminorms gained new interest, when  Bourgain-Brezis-Mironescu \cite{BBM} on one hand, and Maz’ya-Shaposhnikova \cite{MS}  on the other hand showed, that they can be seen as intermediary functionals between the $L^p(\R^N)$-norm and the $W^{1, p}(\R^N)$-seminorms. More precisely, for some given kernel functions (or called mollifiers) $(\rho_n(x, y))_n$ ($\rho_n=\frac{1/n}{\|x-y\|^{N-p/n}}$ for example)  and $(\hat \rho_n(x, y))_n$  ($\hat \rho_n=\frac{1/n}{|x-y|^{N+p/n-p}}$ for example),  there are constants $K={K_{p, N} }$ and $L={L_{p, N} }$ depending only on $p, N$ such that
 \begin{equation}\label{eq1:intro}
 \mathop{\lim}_{n \to \infty} \int_{\R^N} \int_{\R^N} \frac{|f(x)-f(y)|^p}{|x-y|^p} \rho_n(x, y)\,\d\mathcal L^N (x) \d \mathcal L^N (y)= {K }  \| \nabla f\|_{L^p(\R^N)}^p 
\end{equation}
and
 \begin{equation}
 \mathop{\lim}_{n \to \infty}\int_{\R^N} \int_{\R^N}  \frac{|f(x)-f(y)|^p}{|x-y|^p} \hat \rho_n(x, y)\,\d\mathcal L^N (x) \d\mathcal L^N (y)= {L }  \|  f\|_{L^p(\R^N)}^p 
\end{equation}

Recently, there are some new research on Bourgain-Brezis-Mironescu's formula  (BBM formula for short)  \eqref{eq1:intro} in metric measure setting, such as a work of G\'orny \cite{Gorny20} concerning a  Bourgain-Brezis-Mironescu type  formula in metric spaces with Euclidean tangents, and a new characterization of fractional Sobolev functions  obtained by Di Marino-Squassina \cite{DMS19}.  Motivated by these new progress, we realize that the Taylor's formula used in the Euclidean case is not essential. In summary, the key component in the  proof of Bourgain-Brezis-Mironescu's asymptotic formula in metric measure setting are
\begin{itemize}
\item [1)] a non-smooth version of the Rademacher's theorem, i.e. Lipschitz functions are weakly differentiable almost everywhere;
\item [2)] the space (locally)  looks  like a fixed Euclidean space, or more generally, at almost every point, tangent space (in the measured-Gromov-Hausdorff sense) is unique;
\item [3)]the mollifier $\rho_n(x, y)$, is a function depending  only on the distance between $x$ and $y$,   increases  polynomially with some proper order.
\end{itemize}

In this paper,  we focus on the  Bourgain-Brezis-Mironescu's formula  \eqref{eq1:intro} in the setting of abstract metric measure spaces,  namely on the existence and exact value of the constant $K$. We focus on the case $p > 1$ in order to be able to use the equivalence of different definitions of Sobolev spaces and the density of Lipschitz functions in the Sobolev norm. In  another paper \cite{HanPinamonti-MS}, we will prove a generalized Maz’ya-Shaposhnikova's formula in metric measure setting.

\bigskip 

Let us briefly summarize the highlights and main innovations of this paper.
\begin{itemize}
\item [a)] Offers a unified proof to the previous  results,  including  Bourgain-Brezis-Mironescu \cite{BBM} on  $\R^n$, Ludwig \cite{Ludwig14} on  finite dimensional Banach spaces and  Kreuml-Mordhorst \cite{BBM-RM} on  weighted Riemannian manifolds.
\item [b)] Extends the border of Bourgain-Brezis-Mironescu's formula to sub-Riemannian manifolds and $\rcdkn$ metric measure spaces, which has potential applications in geometry, analysis and probability theory.
\item [c)] Our results hold for general mollifiers (see Assumption \ref{assumption2}), and this allows us to prove more Bourgain-Brezis-Mironescu type asymptotic  formulas.  For example, our result can be used to prove the heat semi-group characterization of Sobolev functions (we refer to a recent work of Alonso Ruiz and  Baudoin \cite{BaudionRuiz20} in this direction) .
\end{itemize}

\bigskip

The structure of this paper is as follows. In Section 2 we recall the necessary notions, such as the (equivalent) definitions of Sobolev spaces on a metric measure space, measured-Gromov-Hausdorff convergence and the  Rademacher's theorem on metric measure space. In Section 3, we start by posing the basic assumptions and  proving the formula for bounded  Lipschitz functions,  then we will prove the main result  of the paper Theorem \ref{th1}. In the last section, we will show that Euclidean spaces,  weighted Riemannian  for continuous weights bounded from below,  sub-Riemannian manifolds and non-collapsed $\rcdkn$ spaces fit our setting.

\section{Preliminaries}
Let $(X,\d)$ be a complete, separable and pathwise connected metric space equipped with a locally finite doubling Borel measure $\mm$.
Given $f:\Omega\to \overline{\mathbb{R}}$, where $\Omega\subset X$ is a domain (i.e. connected open set), we call a Borel function $g:\Omega\to [0,\infty]$ an upper gradient of $f$ on $\Omega$ provided
\[
|f(x)-f(y)|\leq \int_{\gamma} g\, \d s
\] 
for all $x,y\in\Omega$ and each rectifiable curve $\gamma:[0,1]\to\Omega$ that joins $x$ and $y$. We say that the metric space $(X,\mm)$ supports a local $p$-Poincar\'e inequality ($1\leq p<\infty$) if for all $K\subset X$ compact there exist $\overline r>0$, $C>0$ and $\lambda\geq 1$ such that 
for all $f$ locally Lipschitz, $x\in K$ and $r\in (0,\overline r)$
\begin{equation}\label{eq:pi}
\fint_{B(x,r)}|f(y)-f_{B}|\, \d\mm(y)\leq Cr \left(\fint_{B(x,\lambda r)} g(y)^p\, \d \mm(y)\right)^{\frac{1}{p}}
\end{equation}
here, as usual, $f_B=\frac{1}{\mm\big(B(x,r)\big)}\int_{B(x,r)} f(y)\, \d \mm(y)=\fint _{B(x,r)}f(y)\, \d\mm(y)$.\\
It is known that  Riemannian manifolds and sub-Riemannian manifolds  support a local Poincar\'e  inequality \cite{AGM}. It is  proved by Rajala \cite{Rajala12} that $\rcdkn$ spaces also support a  local Poincar\'e inequality, and the upper gradient $g$ in \eqref{eq:pi} can be replaced by a weak upper gradient if $u$ is a Sobolev function (cf. the next subsection for the definition).

We say that a metric measure space  $(X,\d,\mm)$ supports a differentiable structure if there is a family $\{U_i, \varphi_i)\}_{i\in \N}$ of Borel charts, i.e. $U_i\subset X$ is a Borel set, $X=\cup_i U_i$ up to an $\mm$-negligible set, and $\varphi_i: X \to \R^{d(i)}$ is Lipschitz such that for every Lipschitz function $f:X\to \mathbb{R}$ is differentiable at $\mm$-a.e. $x_0\in U_i$,  i.e. there is a unique $\d f (x_0)\in \R^{d(i)}$ such that
\begin{equation}\label{cheg}
\lmts{X\ni x}{x_0} \frac{\Big |f(x)-f(x_0)-\d f(x_0)\cdot \big (\varphi(x)-\varphi(x_0)\big)\Big|}{\d(x, x_0)}=0.
\end{equation}
In his celebrated paper \cite{C-D}  Cheeger proved that if $(X,\d,\mm)$ is a doubling metric measure space and supports a Poincar\'e inequality,  then it supports a differentiable structure. 
We conclude this section recalling that in metric measure spaces the boundary area in  the smooth framework can be replaced by the Minkowski content,
\begin{equation}\label{eq:minkow}
\mm^+(E):=\liminf_{\varepsilon\to 0}\frac{\mm(E^{\varepsilon})-\mm(E)}{\varepsilon}
\end{equation}
where $E$ is a Borel set and $E^{\varepsilon}:=\{x\in X \ :\ \exists~y\in E\ \mbox{such that} ~ \d(x,y)<\varepsilon\}$ is the $\varepsilon$-neighbourhood of $E$ with respect to the metric $\d$. 

\subsection{Sobolev spaces on metric measure spaces}
Given $f : X \to \mathbb{R}$, the local Lipschitz constant $\lip{ f}: X \to [0, \infty]$ is defined as
\begin{equation}\label{eq:DeflocLip}
\lip{f}(x):= \mathop{\overline{\lim}}_{y\rightarrow x}\frac{|f(y) -f(x)|}{\d(x, y)}
\end{equation}
if $x$ is not isolated, $0$ otherwise, while the  (global) Lipschitz constant  is defined as
\[
\Lip(f):= \mathop{\sup}_{x \neq y} \frac{|f(y)-f(x)|}{\d(x,y)}.
\]
If $(X,\d)$ is a length space, we know $\Lip(f)=\sup_x \lip{f}(x)$. A function $f$ is Lipschitz and we write $f\in \Lip(X,d)$, if $\Lip(f)<\infty$, and $\Lip_c(X)$ is the collection of Lipschitz $f:X\to \R$, with compact support.

Let $1<p<\infty$. We say that a  functions  $f \in L^p(X,\mm)$  is in  the Sobolev space $W^{1,p}(X,\d,\mm)$  if  there is a sequence of Lipschitz functions $(f_n)$ converging to $f$ in $L^p$, such that 
\[
\liminf_{n \to \infty} \int_X \lip{f_n}^p \,\d\mm<\infty.
\]
Furthermore, there is a unique $L^p$-function   $|\D f|_p$, called  the minimal weak upper gradient, such that 
\begin{eqnarray*}
\int_X\big| \D f|_p^p\,\d\mm= \inf\!\Big\lbrace \liminf_{n \to \infty} \int_X \lip{f_n}^p\d\mm : f_n \in \Lip_{c}(X),\ \! f_n \to f \text{ in } L^p(X,\mm)\Big\rbrace
\end{eqnarray*}
 We refer to \cite{AGS-C} for details. For simplicity, we will neglect the parameter $p$ and denote  $|\D f|_p$ by $|\D f|$. It is known that in many relevant situations like the case of $\rcd$ spaces the value of $|\D f|_p$ is actually independent of $p$.

The  Sobolev space $W^{1,p}\ms$   endowed with the norm
\[
\|f\|^p_{W^{1,p}\ms}:=\|f\|^p_{L^p(X,\mm)}+\||\D f|\|^p_{L^p(X,\mm)}
\]
is always a Banach space, but in general  it  is not a Hilbert space.

We recall the following characterizations of Sobolev functions on metric measure spaces \cite{AGS-D}.

\begin{proposition}[Characterization of Sobolev functions]\label{prop:density}
Let $ f \in  W^{1, p}\ms$ for some $p \geq 1$.  Then there exists a sequence $(f_n)_{n\in\N}\subset \Lip_c(X)$  converging to $f$ in $L^p (X, \mm)$ such that 
\[
\lmt{n}{\infty} \int \big| \lip {f_n} - |\D f| \big|^p \,\d \mm = 0.
\]
\end{proposition}
\begin{remark}
Notice that by Mazur's lemma we can improve this convergence to strong
convergence in $W^{1, p}$, as soon as this space is reflexive; this happens for
instance in the context of $\rcd$ spaces considered in \cite{AGS-M}, with $p= 2$. Or volume doubling spaces supporting a local Poincar\'e inequality,  as proved in \cite{ACD-S}.
\end{remark}

In the proof of the main theorem, we are going to use one more equivalent characterization of Sobolev spaces, called the Hajlasz-Sobolev space $M^{1,p}\ms$ (see Proposition \ref{prop:density2}). While the norms in $W^{1,p}$ and $M^{1,p}$ do not necessarily agree, classical arguments using maximal functions (for instance, combine \cite[Theorem 4.5]{KoskelaMacManus} and \cite[Theorem 1.0.1]{KeithZhong}) imply the following result concerning the equivalence of these spaces.

\begin{proposition}[An equivalent characterization of Sobolev functions]\label{prop:density2}
 Let $p \in (1, \infty)$. Suppose that $\ms$ is a doubling metric measure space supporting a $(1, p)$-Poincar\'e inequality. Then for any $f \in W^{1,p}\ms$, there exists
$g \in  L^p(X, \mm)$ such that 
\[
|f (x) - f (y)| \leq \d(x, y) \big (g(x) + g(y)\big)
\] 
for $\mm$-a.e. $x, y \in X$ (this means that $f$ is in the Hajlasz-Sobolev space $M^{1,p}\ms$).  Moreover,  we can choose $g$ such that $\| g\|_{L^p} \leq C \| \D f\|_{L^p}$ for some universal constant   $C=C(X, p)$.
\end{proposition}

\subsection{Measured Gromov-Hausdorff convergence and tangent cones}



In order to study the convergence of possibly non-compact metric measure spaces, it is useful to fix  reference points. We then say that $(X,\d,\mm,\bar{x})$ is a pointed metric measure space (p.m.m.s. for short), if $(X,\d,\mm)$ is a m.m.s. as before and $\bar{x}\in X$ plays the role of a reference point. For simplicity, we always assume $\supp \mm=X$. We will adopt the following definition of measured Gromov-Hausdorff convergence of  p.m.m.s  (see \cite{BBI-A}, \cite{GMS-C}
 and  \cite{V-O} for equivalent ways to define this convergence).

\begin{definition} Let $\epsilon\in (0,1)$. A map $\varphi : (X_1, \d_1, x_1) \to  (X_2,  \d_2, x_2)$ between two metric spaces with a distinguished point is called an $\epsilon$-isometry if
\begin{itemize}
\item[i)] $\big |\d_2(\varphi (x), \varphi (y)) - \d_1(x, y)\big| \leq  \epsilon$ ~~~ $\forall   x, y \in  B^{X_1}_{\epsilon^{-1}}(x_1)$
\item[ii)]
$B^{X_2}_{r- \epsilon} (x_2) \subset B^{X_2}_\epsilon [\varphi (B^{X_1}_r(x_1))] $~~~$\forall r \in [\epsilon, \epsilon^{-1}]$
\end{itemize}
where $B^{X_i}_\epsilon[A]:=\{z \in X_i: \, \d_i(z, A)<\epsilon\}$ for every subset $A \subset X_i$, $i=1, 2$.
\end{definition}
Note that we do not necessarily ask for  $\varphi(x_1) = x_2$, but the condition ii) implies that $\d_2(\varphi (x_1), x_2)) \leq  2\epsilon$.

Using the above definition we can introduce a notion of convergence for isomorphism classes of p.m.m.s. as follows (see e.g. \cite{GMS-C} for the details).


\begin{definition}[Pointed measured Gromov-Hausdorff convergence]\label{def:conv}
A sequence of pointed metric spaces  $(X_n, \d_n,  \bar{x}_n)$, $n\in \mathbb{N}$   converges in pointed Gromov-Hausdorff sense to $(X, \d, \bar x)$ if there exists a sequence $\epsilon_n \downarrow 0$ such that there exist $\epsilon_n$-isometries  (called Gromov-Hausdorff approximations) $\varphi_n: X_n \to X$ and $\phi_n : X \to  X_n$.

Moreover, we say that a sequence of pointed metric measure spaces $(X_n, \d_n, \mm_n, \bar{x}_n)$  converges in pointed measured  Gromov-Hausdorff sense (p-mGH for short) to a pointed metric measure space $(X,  \d, \mm, \bar x)$, if additionally $(\varphi_n)_\sharp \mm_n \to  \mm$ weakly as measures, i.e. 
\[
\lim_{n \to \infty} \int_X g \, \d \big  ((\varphi_n)_\sharp \mm_n\big) = \int_X g \, \d \mm \qquad \forall g \in C_b(X), 
\]
where $C_b(X)$ denotes the set of real valued bounded continuous functions with bounded support in $X$ and $(\varphi_n)_\sharp \mm_n(A)=\mm_n(\varphi_n^{-1}(A))$ for any $A\subset X_n$ Borel.
\end{definition}

Next, let us recall the notion of measured tangents. Let  $\ms$ be a m.m.s.,  $\bar x\in X$ and $r\in(0,1)$; we consider the rescaled and normalized p.m.m.s. $(X,r^{-1}\d,\mm^{\bar{x}}_r,\bar x)$ where the normalized  measure $\mm^{\bar x}_r$ is given by
\begin{equation}
\label{eq:normalization}
\mm^{\bar x}_r:=\frac1{\mm\big(B_r(\bar x)\big)}\mm\restr{B_r(\bar x)}.
\end{equation}
Note that by definition the unit ball  in $(X,r^{-1}\d,\mm^{\bar{x}}_r,\bar x)$  centred at $\bar{x}$ has volume 1 for any $r>0$.

\begin{definition}[Tangent cone]\label{def:tangent}
Let  $(X,\d,\mm)$ be a m.m.s. and  $\bar x\in X$. A p.m.m.s.  $(Y,\d_Y,\mm_Y,y)$ is called a
\emph{tangent cone} to $(X,\d,\mm)$ at $\bar{x} \in X$ if there exists a sequence of rescalings $r_j \downarrow 0$ so that
$(X,r_j^{-1}\d,\mm^{\bar{x}}_{r_j},\bar{x}) \to (Y,\d_Y,\mm_Y,y)$ as 
$j \to \infty$ in the p-mGH sense.
We denote the collection of all the tangent cones of $(X,\d,\mm)$ at 
$\bar{x} \in X$ by ${\rm Tan}(X,\d,\mm,\bar{x})$.  
\end{definition}

 
\begin{remark}
It is well known (by Gromov's pre-compactness theorem) that on complete metric spaces equipped with a doubling measure,  tangent cones exist for all $x \in X$,  but they are not necessarily unique \cite{Gromov07}.

In case the tangent cone is unique, we will drop the sequence $r_j$ and simply index the blow-ups  by $r \in (0, \infty)$.
\end{remark}

\subsection{Convergence of Lipschitz functions}

We recall various notions of convergence of functions defined on p-mGH converging spaces, see \cite{C-D, GMS-C} for more details.

\begin{definition}[Pointwise and uniform convergence of real-valued functions]\label{def:pointConv}
Let $\big ( (X_n, \d_n, \mm_n, \bar{x}_n) \big )_{n\in \N}$ be a  sequence of  p.m.m.s. p-mGH converging to $(X, \d, \mm, \bar x)$  and let $(f_n)_{n\in \N}$  be a sequence of real-valued functions on $X_n$. We say that $(f_n)_{n\in \N} $ converges pointwisely to a function  $ f$ on $X$ provided
\[
f_n (x_n) \to f(x)~~\forall ~x_n \in X_n ~\text{such that}~\varphi_n (x_n) \to x \text{ in } (X,\d).
\]

If for any $\epsilon>0$ there exists $N\in \mathbb N$ such that  $| f_n(x_n) - f(x)|\leq \epsilon$ for every $n \geq N$ and every $x_n \in X_n, x \in X$ with $\d(\varphi_n (x_n),x) \leq \frac{1}{N}$, we say that $(f_n)_{n\in \N}$ converges  to $f$ uniformly.
\end{definition}

Given $u\in \Lip(X, \d)$, $x\in X$ and $r>0$. 
The rescaling  function $u_{r,x}:X\to \mathbb{R}$ is defined as 
\begin{equation}\label{rescaling}
u_{r,x}(y):=\frac{u(y)-u(x)}r.
\end{equation}
It is easy to see that $u_{r, x}$, $r\geq 0$ are  uniform $\Lip(u)$-Lipschitz functions on $(X, r^{-1} \d)$.  

Assume that $X_{r_i}:=(X,{r_i}^{-1}\d,\mm^{{x}}_{r_i},{x}) \to Y_x:= (Y,\d_Y,\mm_Y,y)$ for some
$r_i \downarrow 0, i\in \N $ in the p-mGH sense (cf. Definition \ref{def:tangent}). Let  $\phi_{r_i}: X_{r_i} \to Y_x,   i\in \N$  be a family of  associated $\epsilon(r_i)$-Gromov-Hausdorff approximation maps with $\epsilon(r) \to 0 $ as $r \downarrow 0$.
By a similar argument  as classical Arzel\`a-Ascoli theorem,  we can show that  the (equi-continuous) rescaling functions $(u_{r_i, x})_{i\in \N}$  converge locally uniformly (up to taking a subsequence and relabelling) to a Lipschitz function function $u_{0,x}$ on $(Y,\d_Y)$,  in the  sense of Definition \ref{def:pointConv}. It can be checked that
this convergence can be equivalently described  as:
for any $R>0$, on $B_R(x)$ we have
\begin{equation}\label{eq:uniform}
\big \|u_{0,x}\circ \phi_{r_i}-u_{r_i,x}\big\|_{L^\infty} \leq  \alpha(r_i),
\end{equation}
where  $ \alpha(r)  \to 0 $ as $r \to 0$.

Concerning  $u_{0, x}$, there is a Rademacher type  theorem proved by  Cheeger in \cite[Theorem 10.2]{C-D}.

\begin{theorem}[Generalized Rademacher theorem]\label{th:rm}
Assume that $\ms$ supports a differentiable  structure. Then  for $\mm$-a.e. $x\in X$ (which are called  points of differentiability of $u$),   any tangent cone $Y_x$ and any limit  $u_{0,x}$ of rescaling functions we have that $u_{0, x}$ is a generalized linear function on $Y_x$. Here by a generalized linear function on $Y$ we mean a Lipschitz function that is harmonic on $Y$ with a constant function as its minimal weak upper gradient (cf. \cite[Definition 8.1]{C-D}).  
\end{theorem}

\begin{remark}
It should be noticed that the tangent cones and Gromov-Hausdorff approximation maps are not unique in general, so the limit of rescaling functions are not necessarily unique. 
Though Cheeger's result does not ensure the uniqueness  of this functions in general, it was proved in \cite[Theorem 10.2]{C-D} that $\Lip u_{0, x}=\lip{u}(x)$, is independent of the particular function $u_{0,z}$ and tangent cone $Y_x$.
\end{remark}

\section{Main results}

\subsection{Basic assumptions}
\begin{assumption}\label{assumption1}
In the rest of this paper, we adopt the following  assumptions on the  metric measure space $\ms$:
\begin{itemize} 
\item [1)](\textbf{Doubling and Poincar\'e}): The measure $\mm$ is doubling and $\ms$ supports a local Poincar\'e inequality.
\item [2a)] (\textbf{Uniqueness of tangent cones}) For  $\mm$-a.e. $x\in X$, the tangent cone is unique,  and isometric to a pointed metric measure space $\mathfrak C:=(C, \d_C, \mm_C, x_0)$. 
\item [2b)] (\textbf{Uniqueness of generalized linear functions}) Given a Lipschitz function $u$.   For  $\mm$-a.e. $x\in X$  and any  two limits of rescaling functions   $u_{0,x}$ and  $\bar u_{0,x}$ (cf. Theorem \ref{th:rm}), then there is an isometry of $\mathfrak C$, denoted by $\varphi$,  such that  $u_{0,x}=\bar u_{0,x}(\varphi)$.
\item [3)] (\textbf{Good Gromov-Hausdorff approximations})  For  $\mm$-a.e. $x\in X$,   there is a family of Gromov-Hausdorff approximations   $(\phi_\delta)_{\delta>0}$ from $(X, \delta^{-1}\d,  x)$ to  $\mathfrak C=(C, \d_C, x_0)$ with $\phi_\delta(x)=x_0$,  an increasing function $\eta(\delta)=\eta(\delta, x): (0, 1) \to \R^+$ with $\lmt{\delta}{0}\eta(\delta)=0$, a measurable  set $\NN\subset X$ and a constant $N\in \mathbb{R}^+$,  such that
\begin{equation}\label{eq0:asmp}
\mm\big(\NN \cap B_\delta(x)\big)  < \eta(\delta) \delta^N
\end{equation} 
and
\begin{equation}\label{eq1:asmp}
\left | \frac{\frac1 \delta {\d(x, y)} }{\d_{ C}(x_0, \phi_\delta(y))  }-1\right | <\eta(\delta),~~~~\forall y\in B_\delta(x)\setminus \Big (\{x\} \cup \NN \Big).
\end{equation}

It can be seen that  \eqref{eq0:asmp}  is equivalent to 
\begin{equation}
\lmt{\delta}{0}\frac{\mm\big(\NN \cap B_\delta(x)\big)}{\delta^N}=0.
\end{equation}

\item [4)]   (\textbf{Uniform approximation of the  measures}) For $\mm$-a.e. $x\in X$, and any $\delta\in (0,1)$ it holds 
\begin{equation}\label{eq2:asmp}
\left |(\phi_\delta)_\sharp \big(\mm^x_\delta\restr{B_\delta(x) \setminus \NN} \big) - \mm_{ C} \restr{\phi_\delta \big(B_\delta(x) \setminus \NN \big) }\right|<\eta(\delta) \mm_{ C } \restr{\phi_\delta \big(B_\delta(x) \setminus \NN \big) }.
\end{equation}

\item [5)] (\textbf{Homogeneity/Self-similarity}) For any $r\in (0,1)$,  there is an isometry  $D_r$ from $\mathfrak C$  to $(C, r^{-1}\d_C, r^{-N}\mm_C, x_0)$ such that $D_r(x_0)=x_0$.  In particular,  
\begin{equation}\label{eq3:asmp}
(D_R)_\sharp \big(\mm_C\restr{ B_1^C(x_0)}\big)=R^{-N}\mm_C\restr{ B_R^C(x_0)}
\end{equation}
 for any $R>0$, where $B_R^C(x_0)$ denotes the ball in $\C$ centred at $x_0$ with radius $R$. In particular, $\mm_C\big( B_1^C(x_0)\big)= R^{-N} \mm_C\big( B_R^C(x_0)\big)$.
\end{itemize}

\end{assumption}
\begin{remark}
Some comments on the previous assumptions are now in order.
\begin{itemize}
\item [a)] By Cheeger's theorem,  Assumption $1)$ implies the existence of a differentiable structure on $X$ and any Lipschitz function $f:X\to \mathbb{R}$ is differentiable at $\mm$-a.e point in the sense of \eqref{cheg}.
\item [b)]  
Assumption $2)$ is satisfied in many relevant situations. 
For example it is satisfied by any oriented Riemannian manifold equipped by any volume measure and even more generally by any Ricci-limit space as proved by Cheeger-Colding  in \cite{CC-O1, CC-O2, CC-O3} and by Colding-Naber in \cite{CN-S}. Remarkably, this property is also satisfied in $\rcdkn$-space (we refer to \cite{AGS-M} for the definition) $(X,\d, \mm)$. Indeed in \cite{GMR-E} it is proved that for $\mm$-a.e. $x\in X$ there exists a blow-up sequence converging to a Euclidean space. The $\mm$-a.e.  uniqueness of the blow-up limit, together with the rectifiability of an $\rcdkn$-space, was then established in \cite{MN-S}. Finally in \cite{BrueSemolaConstant}, it is proved that  if $\ms$ is an $\rcdkn$-space for some $K\in \mathbb{R}$  and $1\leq N<\infty$, then there exists a natural number $1\leq n<N$ such that the tangent cone of $\ms$ is the $n$-dimensional Euclidean space at $\mm$-almost every point in $X$. \\
Assumption $2)$ is also satisfied by any oriented equi-regular sub-Riemannian manifolds equipped by any smooth volume (cf. Proposition \ref{prop:sr}) .
\item[c)] Combining  Lemma \ref{lemma4} and  \eqref{eq0:asmp},  the point $x$ in Assumption 3)  is a density-1 point of $X \setminus \mathcal N$ in $X$. 

\item[d)] By uniqueness of the tangent cones   Assumption $2a)$, it is not hard to see that  $(C, \d_C)$ is isometric to  $(C, r^{-1}\d_C)$. In addition, if $(X,\d)$ is geodesic then \cite[Theorem 1.2]{LeDonne2011}  imply that the metric space $(C,\d_{C})$ is a Carnot group $\mathbb{G}$ endowed with a sub-Finsler left-invariant metric with the first layer of the Lie algebra of $\mathbb{G}$ as horizontal distribution (we refer to  \cite{LeDonne2011} for all the relevant definitions).

\end{itemize}
\end{remark}

Combing Theorem \ref{th:rm} and Assumption \ref{assumption1}, we have the following stronger version of Rademacher-type theorem.
\begin{lemma}[{Rademacher-type theorem}]\label{lemma:rm}
 Given a Lipschitz function $u$. For $\mm$-a.e. $x\in X$,  there is a family of good Gromov-Hausdorff approximations from $(X, \delta^{-1}\d,  x)$ to  $\mathfrak C=(C, \d_C, x_0)$ satisfying Assumption \ref{assumption1},  still denoted by   $(\phi_\delta)_{\delta>0}$,  such that the rescaling functions   $(u_{r, x})_{r>0}$ has a unique uniform limit $u_{0, x}$ in the  sense of \eqref{eq:uniform}, i.e. there is $ \alpha(r) $  converging to 0  as $r \to 0$ such that 
\begin{equation}\label{eq2:uniform}
\big \|u_{0,x}(\phi_{r}(\cdot)) -u_{r,x}(\cdot)\big\|_{L^\infty} \leq  \alpha(r),~~\text{on}~B_R(x), ~\forall R>0.
\end{equation}

\end{lemma}
\begin{proof}
Let  $(\phi_\delta)_{\delta>0}$ be  the good Gromov-Hausdorff approximations given in Assumption \ref{assumption1} and $(r_i)_{i\in \N}$, $(\bar r_i)_{i\in \N}$ be any two sequences of positive numbers converging to 0,  such that the rescaling functions $(u_{r_i, x})_{i\in \N}$ and $(u_{\bar r_i, x})_{i\in \N}$  converge to  generalized linear functions $u_{0,x}$  and  $\bar u_{0,x}$ respectively, i.e.
\[
\big \|u_{0,x}(\phi_{r_i}(\cdot)) -u_{r_i,x}(\cdot)\big\|_{L^\infty} \to 0
\]
and
\[
\big \|\bar u_{0,x}(\phi_{\bar r_i}(\cdot)) -u_{\bar r_i,x}(\cdot)\big\|_{L^\infty} \to 0.
\]

By Assumption \ref{assumption1}-2b), there is an isometry $\varphi$ of $\mathfrak{C}$,  such that  $\bar u_{0,x}= u_{0,x}(\varphi)$. So
\[
\big \|u_{0,x}(\varphi \circ \phi_{\bar r_i}(\cdot)) -u_{\bar r_i,x}(\cdot)\big\|_{L^\infty} \to 0.
\]
Then we  replace $(\phi_{\bar r_i})_{i}$ by $(\varphi \circ \phi_{\bar r_i} )_{i}$,  so that  both $(u_{r_i, x})_{i\in \N}$ and $(u_{\bar r_i, x})_{i\in \N}$ converge to the   limit function $ u_{0,x}$ w.r.t. these new Gromov-Hausdorff approximations. 

Repeat the argument above, we  find a desirable sequence of good Gromov-Hausdorff approximations such that the convergence \eqref{eq2:uniform} hold.
\end{proof}

\bigskip
Let  $p>1$ and $u\in W^{1, p}\ms$. The aim of this paper is to study the following limit
\begin{equation}\label{eq0:th1}
\mathop{\lim}_{n \to \infty} \int_X \int_X \frac{|u(x)-u(y)|^p}{\d^{p}(x, y)} \rho_n(x, y)\,\d\mm(x) \d\mm(y)
\end{equation}
where $(\rho_n)_{n \in \N}$ is a sequence of mollifiers $\rho_n:X\times X\to [0,\infty)$ satisfying  the following set of assumptions:
\begin{assumption}\label{assumption2}
\item [1')] (\textbf{Uniform boundedness}): 
\[
c_1:=\sup_{x\in X} \sup_{n\in \N} \|\rho_n(x, \cdot)\|_{L^1} <\infty;
\]
\item [2')] (\textbf{Radial distribution}): There is a sequence of non-negative non-increasing functions $({\tilde \rho}_n)_{n\in \N}$ such that
\[
\rho_n (x, y)=\tilde\rho_n \big( \d(x, y)\big)~~~\forall x, y\in X.
\]
\item [3')] (\textbf{Polynomial decay at infinity}):  For  any $\delta>0$, it holds the limit
\[
 \lmt{n}{\infty}  \int_0^\delta  r^{N-1} \tilde \rho_n(r)\,\d r=\Big(\mm_C^+(S^C_1)\Big )^{-1}
\]
where $S^C_r$ denotes the boundary of the  ball  $ B^{C}_r$  in the tangent cone $\mathfrak{ C}$ with radius $r$, $N\in \mathbb{R}^+$ is the constant appearing in Assumption \ref{assumption1}-3) and $\mm_C^+$ is the canonical  Minkowski content  w.r.t. $\mm_{C}$ (cf. \eqref{eq:minkow})  which is well-defined on the Borel subsets of $S_r^{C}$ (we will see that $\mm_C^+(S^C_1)=N$ in Lemma \ref{lemma4}).

 For  any $\delta>0$,  $x\in X$, it also holds
 \[
 \lmt{n}{\infty}\int_{(B_\delta(x))^c}   \rho_n(x, y)\,\d \mm(y)=0.
 \]
\end{assumption}

 \subsection{The proofs}
We will show that the limit in \eqref{eq0:th1} is
 \[
\| \nabla u\|^p_{K_{p, \mathfrak C} }:=\int \left(\lmt{\delta}{0} \Big( \frac{\mm(B_\delta(x))}{\delta^N} \Big ) \fint_{S_1^{C} } |u_{0, x}(v) |^p  \,\d\mm^+_{C}(v)\right)\,\d \mm(x).
 \]
 Note that by uniqueness of generalized linear function (cf. Assumption \ref{assumption1}-21) and Lemma \ref{lemma:rm}),  the integration of  $u_{0, x}$ on $S_1^{C} $  makes sense and  $\fint_{S_1^{C} } |u_{0, x}(v) |^p  \,\d\mm^+_{C}(v)$ is well-defined.

 First of all, we prove  the following basic estimate.
 
 \begin{proposition}\label{prop1}
Let $p\in [1, \infty)$, $u\in \Lip_c(X, \d)$ and let $(\rho_n)_{n\in\mathbb{N}}$ be a sequence of mollifiers satisfying Assumption \ref{assumption2}. Then for any  $\delta>0$,  it  holds
 \begin{eqnarray*}
&& \E_n(u):=\int_X \int_X  \frac{|u(x)-u(y)|^p}{\d^{p}(x, y)} \rho_n(x, y)\,\d\mm(x) \d\mm(y)\\
 & \leq& c_1\|{\Lip_\delta(u )}\|^p_{L^p} +C_1\| u\|^p_{L^{p}}
 \end{eqnarray*}
 where $c_1=\sup_{x\in X} \sup_{n\in\mathbb{N}} \|\rho_n(x, \cdot)\|_{L^1} $, $C_1= \frac{2^{p} c_1} {\delta^{p}} $ and $$\Lip_\delta(u)(x):=\sup_{y\in B_\delta(x)} \frac{|u(x)-u(y)|}{\d(x, y)}.$$
 \end{proposition}

 \begin{proof}
Fix $x\in X$ and $\delta>0$.  We have the following estimate
  \begin{eqnarray*}
 &&  \int_{ B_\delta(x)} \frac{|u(x)-u(y)|^p}{\d^{p}(x, y)} \rho_n(x, y)\,\d\mm(y)  \\
  &\leq&  \int_{ B_\delta(x)} \sup_{y\in B_\delta(x)}  \frac{|u(x)-u(y)|^p}{\d^{p}(x, y)} \rho_n(x, y)\,\d\mm(y)  \\
  &\leq& \int_X  {|\Lip_\delta( u)(x)|^p} \rho_n(x, y)\, \d\mm(y) \\
  &\leq &  c_1 {|\Lip_\delta( u)(x)|^p}.
 \end{eqnarray*}
   Thus
\begin{equation}\label{eq1:prop1}
  \int_{\{(x, y) \in X \times X, \d(x, y)<\delta\}}\frac{|u(x)-u(y)|^p}{\d^{p}(x, y)} \rho_n(x, y)\,\d\mm(x)\d\mm(y) \leq c_1 \|{\Lip_\delta( u)}\|^p_{L^p(X, \mm)}.
\end{equation}
Combining with the monotonicity of $\rho_n$ (c.f. Assumption \ref{assumption2}),  we have
 \begin{eqnarray*}
 && \int_X \int_X \frac{|u(x)-u(y)|^p}{\d^{p}(x, y)} \rho_n(x, y)\,\d\mm(x) \d\mm(y) \\
 &=&\int_{\{(x, y) \in X\times X, \d(x, y)<\delta\}\cup {\{(x, y) \in X\times X, \d(x, y)\geq \delta\}}}\frac{|u(x)-u(y)|^p}{\d^{p}(x, y)} \rho_n(x, y)\,\d\mm(x) \d\mm(y) \\
  &\leq& c_1\|{\Lip_\delta( u)}\|^p_{L^p(X)}+
    \int_{x\in X} \int_{\{y\in X, \d(y, x)\geq \delta\} }\frac{|u(x)-u(y)|^p}{\delta^{p}} \rho_n(x, y)\,  \d\mm(x)\mm(y)\\
     &\leq& c_1\|{\Lip_\delta( u)}\|^p_{L^p(X)}+
   2^{p-1} \int_X \int_{X}\frac{|u(x)|^p+|u(y)|^p}{\delta^{p}} \rho_n(x, y)\,  \d\mm(x)\mm(y)\\
    &\leq&c_1 \|{\Lip_\delta( u)}\|^p_{L^p(X)}+
   \frac{2^{p} c_1} {\delta^{p}}  \| u\|^p_{L^p(X)}
 \end{eqnarray*}
 which is the thesis.

 \end{proof}

 Before proving the main result in this section, we   prove the following lemma concerning the density function.

\bigskip
\begin{lemma}\label{lemma4}
The upper and lower density functions $\theta^\pm$  are defined as
\[
\theta^+(x):=\lmts{\delta}{0} \Big( \frac{\mm(B_\delta(x))}{\delta^N} \Big )
\]
and
\[
\theta^-(x):=\lmti{\delta}{0} \Big( \frac{\mm(B_\delta(x))}{\delta^N} \Big ).
\]
Under  the Assumption \ref{assumption1}, we have
\[
\theta^+=\theta^-=\lmt{\delta}{0} \Big( \frac{\mm(B_\delta(x))}{\delta^N} \Big )=:\theta(x)<+\infty~~~~\mm-\text{a.e.}
\]
and
\[
\mm_C^+(S^C_1)=N.
\]
 \end{lemma}
\begin{proof}
Given  $x\in X$ and $\mathcal N\subset X$ satisfying the Assumption \ref{assumption1}.
Denote $\NN_\delta:=\NN \cap B_\delta(x)$. By Assumption \ref{assumption1},  for any $\lambda >1$,  we have
\begin{eqnarray*}
&&\frac{\mm\big(B_\delta(x)\big)}{\delta^N} \\
&=& \frac{\mm\big(B_\delta(x)\setminus \NN_\delta\big)+\mm(\NN_\delta)}{\delta^N}\\
\text {By}  ~\eqref{eq0:asmp} &=& o(1)+\frac{\mm\big(B_\delta(x)\setminus \NN_\delta\big)}{\mm\big (B_{\lambda \delta}(x)\big)} \frac {\mm\big (B_{\lambda \delta}(x)\big)} {\delta^N} \\
\text {By}  ~\eqref{eq2:asmp}&=&o(1)+\big(1+o(1)\big)\mm_C\Big(\phi_{\lambda \delta}\big(B_\delta(x)\setminus \NN_\delta\big)\Big) \frac {\mm\big (B_{\lambda \delta}(x)\big)} {\delta^N} \\
&\leq&o(1)+\big(1+o(1)\big)\mm_C\Big(\phi_{ \lambda\delta}\big(B_\delta(x)\big)\Big)\frac {\mm\big (B_{\lambda \delta}(x)\big)} {\delta^N}\\
\end{eqnarray*}

By   \eqref{eq3:asmp} in Assumption \ref{assumption1},  we have the homogeneity of $\mm_C$: $ \mm_C\big(B^C_r(x_0)\big)=r^N \mm_C\big(B^C_1(x_0)\big)=r^N$ for any $r>0$, so $\mm_C(S_r^C(x_0))=0$. Combining with  the fact that $\phi_{\lambda \delta}(B_\delta (x)) \to B_{\lambda^{-1}}^C(x_0)$ in Hausdorff distance as $\delta \to 0$, we know
\[
 \mm_C\Big(\phi_{ \lambda\delta}\big(B_\delta(x)\big)\Big)=\mm_C\Big(B_{\lambda^{-1}}^C(x_0)\Big)+o(1)=\lambda^{-N}+o(1)~~\text{as}~\delta\to 0.
\]

In conclusion, as $\delta \to 0$  we have
\begin{equation}\label{eq1:lemma4}
\frac{\mm\big(B_\delta(x)\big)}{\delta^N}\leq o(1)+\big(1+o(1)\big)\frac {\mm\big (B_{\lambda \delta}(x)\big)} {(\lambda \delta)^N}.
\end{equation}
In particular, $\theta^-(x) \leq \theta^+(x)< +\infty$.   

For any $\epsilon>0$, there are  $0<\delta_0<\bar \delta_0<\delta \ll 1$  such that 
\[
\frac{\mm\big(B_{\delta_0}(x)\big)}{\delta_0^N}>\theta^+(x)-\epsilon,
\]
and
\[
\frac{\mm\big(B_{\bar \delta_0}(x)\big)}{\bar \delta_0^N}<\theta^-(x)+\epsilon.
\]
Therefore, by \eqref{eq1:lemma4},  we can pick $\delta$ small enough such that  the following inequalities  hold
\[
\theta^+(x)-\epsilon<\frac{\mm\big(B_{\delta_0}(x)\big)}{\delta_0^N} \leq \epsilon+(1+\epsilon) \frac{\mm\big(B_{\bar \delta_0}(x)\big)}{\bar \delta_0^N}     \leq \epsilon+(1+\epsilon) \big( \theta^-(x) +\epsilon).
\]
Letting $\epsilon \to 0$, we get
\[
\theta^-(x)=\theta^+(x).
\]

\bigskip

At last, note that $\mm_C\Big(B_1^C (x_0)\Big)=1$,  by homogeneity (c.f. Assumption \ref{assumption1}-\eqref{eq3:asmp})
\begin{eqnarray*}
1&=& \mm_C\Big(B_1^C (x_0)\Big)\\
&=& \int_0^1 \mm^+_C\Big(S_r^C (x_0)\Big) \,\d r\\
&=& \int_0^1 r^{N-1}\mm^+_C \Big(S_1^C (x_0)\Big) \,\d r\\
&=&\frac1N  \mm^+_C \Big(S_1^C (x_0)\Big).
\end{eqnarray*}
Hence
\[
\mm_C^+(S^C_1)=N.
\]
\end{proof}
\bigskip

Next  in Proposition \ref{prop2} we will prove a generalized  Bourgain-Brezis-Mironescu's formula  for  Lipschitz functions with compact support. This is one of the most important results in this paper. 
 
 \begin{proposition}[Generalized  Bourgain-Brezis-Mironescu's formula]\label{prop2}
 Given $p\geq 1$ and $u\in  \Lip_c(X, \d)$.  For any $n\in \N$, define 
 \[\E_n(u):=\int_X \int_X  \frac{|u(x)-u(y)|^p}{\d^{p}(x, y)} \rho_n(x, y)\,\d\mm(x) \d\mm(y).
 \]
 It holds
\begin{equation}\label{eq0:coro1}
\lmt{n}{\infty} \E_n(u) =\| \nabla u\|^p_{K_{p, \mathfrak C}}\leq  \int \theta |\nabla u|^p\,\d \mm,
\end{equation}
where 
\[
\| \nabla u \|^p_{K_{p, \mathfrak C} }:=\big \|  | \nabla u|_{K_{p, \mathfrak C} }\big\|_{L^1(X, \mm)}
\]
with
\[
| \nabla u|_{K_{p, \mathfrak C} }:=\theta(x) \fint {S_1^{C} }   \left | u_{0,x}(v)\right| ^p\,\d\mm^+_{C}(v)
\]
and $\theta$ is the density function given in Lemma \ref{lemma4}.
 \end{proposition}

 \begin{proof}

{\bf Step 1}:  

Let $(\phi_\delta)$ be a family of Gromov-Hausdorff approximations given  in Assumption \ref{assumption1}, and  $u_{0, x}$  be the uniform limit   of the rescaling  functions $u_{\delta, x}(y):=\frac{u(y)-u(x)}\delta $ as  $\delta \downarrow 0$ (by assumption, this function is unique). Combining \eqref{eq:uniform}, \eqref{eq1:asmp} and the fact  that  $(u_{r,x})_r$ are uniform $\Lip(u)$-Lipschitz functions (w.r.t. the metric $ \d/ r$), there is an increasing function  $\alpha(r)$ satisfying  $\alpha(r) \to 0 $ as $r \downarrow  0$, such that 
\begin{equation}\label{eq0.2:prop2}
\big| u_{0,x}(\phi_r(y)) -u_{r,x}(y)\big| \leq  A(r):=\min\big\{ \alpha(r), 3r^{-1}\d(x, y)\Lip(u)\big\}
\end{equation}
for enough small $r>0$ and any $y\in B_r(x)\setminus \NN$.   Denote $\delta_i:={2^{-i}\delta},  i=0, 1,...,m,...$.  Then we can write
\begin{eqnarray*}
&&\int_{B_\delta(x) }\frac{|u(x)-u(y)|^p}{\d^{p}(x, y)}\rho_n(x, y) \,\d\mm(y)\\
&=& \sum_{i=0}^{\infty} \delta_i^{p} \int_{ B_{\delta_{i}}(x)\setminus \big( B_{\delta_{i+1}}(x) \cup \NN\big)}\frac{|u(x)-u(y)|^p}{\delta_i^{p}} \frac{\rho_n(x, y)}{\d^{p}(x, y)} \,\d\mm(y)\\
&&+\sum_{i=0}^{\infty}  \int_{\big( B_{\delta_{i}}(x)\setminus B_{\delta_{i+1}}(x)\big) \cap \NN}\frac{|u(x)-u(y)|^p}{\d^{p}(x, y)}\rho_n(x, y) \,\d\mm(y)
\end{eqnarray*}
Combining with \eqref{eq0.2:prop2}, we get:
\begin{eqnarray*}
&&\int_{B_\delta(x) }\frac{|u(x)-u(y)|^p}{\d^{p}(x, y)}\rho_n(x, y) \,\d\mm(y)\\
&\leq&\sum_{i=0}^{\infty}\underbrace{ \delta_i^{p} \int_{ B_{\delta_{i}}(x)\setminus \big( B_{\delta_{i+1}}(x) \cup \NN\big)}\left(\big | u_{0,x}(\phi_{\delta_{i}}(y))\big|^p\right ) \frac{\rho_n(x, y)}{\d^{p}(x, y)}  \,\d\mm(y)}_{I(i, \delta, n)}\\
&&+\underbrace{\sum_{i=0}^{\infty} \delta_i^{p} \int_{ B_{\delta_{i}}(x)\setminus B_{\delta_{i+1}}(x)}O\big(A(\delta_i) \big)\frac{\rho_n(x, y)}{\d^{p}(x, y)}  \,\d\mm(y)}_{II(\delta, n)}\\
&&+ |\Lip{(u)}|^p \underbrace{\sum_{i=0}^{\infty} \int_{\big(B_{\delta_i}(x)\setminus B_{\delta_{i+1}}(x)\big)  \cap \NN} {\rho_n (x, y)}\,\d\mm(y)}_{III(\delta,  n)},
\end{eqnarray*}


We will estimate the above quantities $I, II, III$ in the next  two steps.

\bigskip

{\bf Step 2}:  
By \eqref{eq1:asmp} in Assumption \ref{assumption1}   we  have:
\begin{equation}\label{eq1:prop2}
\delta^{-1} \d(x, y)= \d_C\big (x_0, \phi_\delta(y)\big)\big (1+o(1)\big)
\end{equation}
for any $y\in B_\delta(x) \setminus \NN$.

 By Assumption \ref{assumption2},  $\tilde \rho_n$ is non-increasing.   Combining with \eqref{eq1:prop2}, it holds the following estimate for  any given $\epsilon>0$ provided $\delta$ is small enough: 
\[
 \frac{\tilde \rho_n\big(\delta \d_{C}(x_0, \varphi_\delta (y))(1+\epsilon)\big)}{\big(\delta \d_{C}(x_0, \varphi_\delta (y))(1+\epsilon)\big)^{p}} \leq  \frac{\tilde \rho_n( \d(x, y))}{\d^{p} (x, y)}\leq \frac{\tilde \rho_n\big(\delta \d_{C}(x_0, \varphi_\delta (y))(1-\epsilon)\big)}{\big(\delta \d_{C}(x_0, \varphi_\delta (y))(1-\epsilon)\big)^{p}}.
\]
For simplicity,  for small $\delta$, we will formally write 
\begin{equation}\label{eq1.1:prop2}
 \frac{\tilde \rho_n( \d(x, y))}{\d^{p} (x, y)}= \frac{\tilde \rho_n\big(\delta \d_{C}(x_0, \varphi_\delta (y))(1\pm \epsilon)\big)}{\big(\delta \d_{C}(x_0, \varphi_\delta (y))(1\pm \epsilon)\big)^{p}}.
\end{equation}

Hence for   small $\delta$, we  have
\begin{eqnarray*}
&& I(i, \delta, n) \\&=&\delta_i^{p} \int_{ B_{\delta_i}(x)\setminus  \big( B_{\delta_{i+1}}(x) \cup \NN\big)}\big | u_{0,x}(\phi_{\delta_i}(y))\big|^p \frac{\rho_n(x, y)}{\d^{p}(x, y)}  \,\d\mm(y)\\ 
&=&\delta_i^{p} \int_{ B_{\delta_i}(x)\setminus  \big( B_{\delta_{i+1}}(x) \cup \NN\big)} \big | u_{0,x}(\phi_{\delta_i}(y))\big|^p  \frac{\tilde \rho_n\big(\delta_i\d_{C}(x_0, \varphi_{\delta_i} (y))(1\pm \epsilon)\big)}{\big(\delta_i \d_{C}(x_0, \varphi_{\delta_i} (y))(1\pm \epsilon)\big)^{p}} \,\d\mm(y)\\
&= &\delta_i^{p} {\mm\big (B_{\delta_i} (x)\big )} \int_{\phi_{\delta_i}\Big( B_{\delta_i}(x)\setminus  \big( B_{\delta_{i+1}}(x) \cup \NN\big) \Big)}\big | u_{0,x}(v)\big|^p \frac{\tilde \rho_n\big(\delta_i \d_{C}(x_0, v)(1\pm \epsilon)\big)}{\big(\delta_i \d_{C}(x_0, v)(1\pm \epsilon)\big)^{p}}\,\d (\phi_{\delta_i})_\sharp \mm^x_{\delta_i} (v)\\
&=& \delta_i^{p} {\mm\big (B_{\delta_i} (x)\big )} (1 \mp \epsilon) \int_{\phi_{\delta_i}\Big( B_{\delta_i}(x)\setminus  \big( B_{\delta_{i+1}}(x) \cup \NN\big) \Big)} \big | u_{0,x}(v)\big|^p \frac{\tilde \rho_n\big(\delta_i \d_{C}(x_0, v)(1\pm \epsilon)\big)}{\big(\delta_i \d_{C}(x_0, v)(1\pm \epsilon)\big)^{p}}\,\d \mm_{C}(v)
\end{eqnarray*}
where the last (in)equality follows from  \eqref{eq2:asmp} in the Assumption \ref{assumption1}.

Keep   the convention we adopted in \eqref{eq1.1:prop2}, we divide the quantity $I(i, \delta, n) $ into the following three integrals:
\begin{eqnarray*}
&&I(i, \delta, n) \\
&=& \underbrace{ \delta_i^{p} {\mm\big (B_{\delta_i} (x)\big )} (1 \mp \epsilon) \int_{B_{1}^{C}\setminus B_{\frac{1}2}^C}  \big | u_{0,x}(v)\big|^p \frac{\tilde \rho_n\big(\delta_i \d_{C}(x_0, v)(1\pm \epsilon)\big)}{\big(\delta_i \d_{C}(x_0, v)(1\pm \epsilon)\big)^{p}}\,\d \mm_{C}(v)}_{I_a}\\
&+&\underbrace{ \delta_i^{p} {\mm\big (B_{\delta_i} (x)\big )} (1\mp \epsilon) \int_{\phi_{\delta_i}\Big( B_{\delta_i}(x)\setminus  \big( B_{\delta_{i+1}}(x) \cup \NN\big) \Big)\setminus {\big( B_{1}^{C}\setminus B_{\frac{1}2}^C} \big)} \big | u_{0,x}(v)\big|^p \frac{\tilde \rho_n\big(\delta_i \d_{C}(x_0, v)(1\pm \epsilon)\big)}{\big(\delta_i \d_{C}(x_0, v)(1\pm \epsilon)\big)^{p}}\,\d \mm_{C}(v)}_{I_b}\\
&-&\underbrace{ \delta_i^{p} {\mm\big (B_{\delta_i} (x)\big )} (1\mp \epsilon) \int_{{\big( B_{1}^{C}\setminus B_{\frac{1}2}^C} \big) \setminus \phi_{\delta_i}\Big( B_{\delta_i}(x)\setminus  \big( B_{\delta_{i+1}}(x) \cup \NN\big) \Big)} \big | u_{0,x}(v)\big|^p \frac{\tilde \rho_n\big(\delta_i \d_{C}(x_0, v)(1\pm \epsilon)\big)}{\big(\delta_i \d_{C}(x_0, v)(1\pm \epsilon)\big)^{p}}\,\d \mm_{C}(v)}_{I_c}.
\end{eqnarray*}

By change of variable and \eqref{eq3:asmp} in Assumption \ref{assumption1}, we have
\begin{eqnarray*}
&&I_a(i, \delta, n)  \\
&=& {\delta_i^{p}\frac {\mm\big (B_{\delta_i} (x)\big )} {\delta_i^N} \frac {1\mp \epsilon}{(1\pm \epsilon)^N}\int_{B_{\delta_i(1\pm \epsilon)}^{C}\setminus B_{\delta_{i+1}(1\pm \epsilon)}^C}  \big | u_{0,x}\big(D_{\delta_i^{-1}(1\pm \epsilon)^{-1}}(v)\big )\big|^p \frac{\tilde \rho_n\big(\d_{C}(x_0, v)\big)}{\big( \d_{C}(x_0, v)\big)^{p}} \,\d\mm_{C}(v)}\\
&=& {\delta_i^{p}\frac {\mm\big (B_{\delta_i} (x)\big )} {\delta_i^N} \frac {1\mp \epsilon}{(1\pm \epsilon)^N}\int^{\delta_i(1\pm \epsilon)}_{\delta_{i+1}(1\pm \epsilon)}\int_{S_{r}^{C}(x_0)}  \big | u_{0,x}\big(D_{\delta_i^{-1}(1\pm \epsilon)^{-1}}(v)\big )\big|^p \frac{\tilde \rho_n\big(\d_{C}(x_0, v)\big)}{r^{p}} \,\d\mm^+_{C}(v)} \,\d r\\
&=& {\delta_i^{p}\theta^\mp(x) \frac {(1\mp \epsilon)^2}{(1\pm \epsilon)^N}\int^{\delta_i(1\pm \epsilon)}_{\delta_{i+1}(1\pm \epsilon)}\int_{S_{r}^{C}(x_0)}  \big | u_{0,x}\big(D_{\delta_i^{-1}(1\pm \epsilon)^{-1}}(v)\big )\big|^p \frac{\tilde \rho_n\big(\d_{C}(x_0, v)\big)}{r^{p}} \,\d\mm^+_{C}(v)} \,\d r
\end{eqnarray*}
where $D_\cdot$ is the isometry map  in Assumption \ref{assumption1}-5).
Following the convention in \eqref{eq1.1:prop2}, the estimate above  is  understood as
\begin{equation}\label{eq1.2:prop2}
I_a(i, \delta, n)  \leq  \theta^+(x)\frac {(1+\epsilon)^2}{(1-\epsilon)^N}\int^{\delta_i(1- \epsilon)}_{\delta_{i+1}(1-\epsilon)} \int_{S_r^{C} }  \left(\frac{ \delta_i\Big | u_{0,x}\big(D_{\delta_i^{-1}(1-\epsilon)^{-1}}(v)\big )\Big| }{r}\right)^p \tilde \rho_n (r)\,\d\mm^+_{C}(v)\,\d r
\end{equation}
and
\begin{equation}\label{eq1.3:prop2}
I_a (i, \delta, n) \geq  \theta^-(x)\frac {(1-\epsilon)^2}{(1+\epsilon)^N}\int^{\delta_i(1+ \epsilon)}_{\delta_{i+1}(1+ \epsilon)} \int_{S_r^{C} }  \left(\frac{ \delta_i\Big | u_{0,x}\big(D_{\delta_i^{-1}(1+\epsilon)^{-1}}(v)\big )\Big| }{r}\right)^p \tilde \rho_n (r)\,\d\mm^+_{C}(v)\,\d r
\end{equation}


\bigskip

{\bf Step 3}:

Given $\epsilon >0$.  By Lemma \ref{lemma1},  Lemma \ref{lemma2},  Lemma \ref{lemma3} and Lemma \ref{lemma6}, there exists $\delta>0$, such that the following (uniform) estimates hold for any $n\in \N$
\begin{equation}\label{eq1.4:prop2}
\mathop{\sum}_{i=0}^{\infty}I_a(i, \delta,  n)  \leq  \theta^+(x)\big (1+O(\epsilon)\big)\left( \int_{0}^{(1-\epsilon)\delta}  r^{N-1}  \tilde\rho_n (r) \,\d r \right)\left(\int_{S_1^{C} (x_0)}   \left | u_{0,x}(v)\right| ^p\,\d\mm^+_{C}(v)\right),
\end{equation}
\begin{equation}\label{eq2:prop2}
\mathop{\sum}_{i=0}^{\infty}\big(I_b(i, \delta, n)+I_c(i, \delta, n)\big)<\epsilon,
\end{equation}
\begin{equation}\label{eq2.1:prop2}
   II(\delta, n)<\epsilon,
\end{equation}
and
\begin{equation}\label{eq3:prop2}
  III(\delta,  n)<\epsilon.
\end{equation}

For any $\delta\in (0, 1)$, by Assumption \ref{assumption2},   it holds the following estimate:
\begin{equation}\label{eq0.1:prop2}
\lmt{n}{\infty} \int_{B^c_\delta(x)} \frac{|u(x)-u(y)|^p}{\d^{p}(x, y)}\rho_n(x, y) \,\d\mm(y) \leq \lmt{n}{\infty} \Lip(u)^p \int_{B^c_\delta(x)} \rho_n(x, y) \,\d\mm(y)=0.
\end{equation}

Combining \eqref{eq1.4:prop2}, \eqref{eq2:prop2}, \eqref{eq2.1:prop2}, \eqref{eq3:prop2}  and \eqref{eq0.1:prop2}, for $\delta$ small, we have
\begin{eqnarray*}
&&\lmts{n}{\infty}\int_{X }\frac{|u(x)-u(y)|^p}{\d^{p}(x, y)}\rho_n(x, y) \,\d\mm(y)\\ 
&=&\lmts{n}{\infty}\Big (\int_{B_\delta(x) }\frac{|u(x)-u(y)|^p}{\d^{p}(x, y)}\rho_n(x, y) \,\d\mm(y)+\int_{B^c_\delta(x) }\frac{|u(x)-u(y)|^p}{\d^{p}(x, y)}\rho_n(x, y) \,\d\mm(y)\Big)\\ 
&= & \lmts{n}{\infty}\int_{B_\delta(x) }\frac{|u(x)-u(y)|^p}{\d^{p}(x, y)}\rho_n(x, y) \,\d\mm(y)\\
&=&  \lmts{n}{\infty}\Big (I(\delta, n)+ II(\delta, n)+ III(\delta,  n)\Big)\\
&\leq&\lmts{n}{\infty}  \left (\theta^+(x)\big (1+O(\epsilon)\big)\Big( \int_{0}^{\delta}  r^{N-1}  \tilde\rho_n (r) \,\d r \Big)\Big(\int_{S_1^{C} }   \left | u_{0,x}(v)\right| ^p\,\d\mm^+_{C}(v)\Big) +3\epsilon\right)\\
&=&\theta^+(x)\big (1+O(\epsilon)\big) {\Big(\mm_C^+(S^C_1)\Big )^{-1}\Big(\int_{S_1^{C} }   \left | u_{0,x}(v)\right| ^p\,\d\mm^+_{C}(v)\Big)}+3\epsilon.
\end{eqnarray*}

Letting $\epsilon \to 0$, the above estimates implies

\begin{eqnarray*}
&&\lmts{n}{\infty}\int_{X }\frac{|u(x)-u(y)|^p}{\d^{p}(x, y)}\rho_n(x, y) \,\d\mm(y)\\ 
&\leq&\theta^+(x) \Big(\mm_C^+(S^C_1)\Big )^{-1}\Big(\int_{S_1^{C} }   \left | u_{0,x}(v)\right| ^p\,\d\mm^+_{C}(v)\Big)\\
&=&\underbrace{ {\theta^+(x)\fint_{S_1^{C} }   \left | u_{0,x}(v)\right| ^p\,\d\mm^+_{C}(v)}}_{| \nabla u|_{K^+_{p, \mathfrak C} }}.
\end{eqnarray*}

In addition,  note that $(u_{r, x})_{r>0}$ are  $\Lip_r (u)(x)$-Lipschitz on $\big( B_r(x), \d/r\big)$, so by definition $u_{0, x}$ is a $|\lip{u}(x)|$-Lipschitz function. Combining  with the conclusions obtained  above,    we  can also see that 
\begin{eqnarray*}
&&{| \nabla u|_{K^+_{p, \mathfrak C} }}\\
&\leq &  \theta^+(x) \fint_{S_1^{C} }  |\lip{u}(x)| ^p\,\d\mm^+_{C}(v)\\
&= &  \theta^+(x)  | \lip{u}(x)|^p.
\end{eqnarray*}
Thus
\begin{eqnarray*}
&&\lmts{n}{\infty}\int_{X }\frac{|u(x)-u(y)|^p}{\d^{p}(x, y)}\rho_n(x, y) \,\d\mm(y)\\ 
&\leq& \theta^+(x)  | \lip{u}(x)|^p.
\end{eqnarray*}

\bigskip

{\bf Step 4}:

Similarly, we can prove
\begin{eqnarray*}
&&\lmti{n}{\infty}\int_{X }\frac{|u(x)-u(y)|^p}{\d^{p}(x, y)}\rho_n(x, y) \,\d\mm(y)\\ 
&\geq&\underbrace{ \theta^-(x)\fint_{S_1^{C} }   \left | u_{0,x}(v)\right| ^p\,\d\mm^+_{C}(v)}_{| \nabla u|_{K^-_{p, \mathfrak C} }}.
\end{eqnarray*}

Integrating the inequalities obtained in {\bf Step 3} and  {\bf Step 4}  above,  by Fatou's lemma we prove
\begin{equation}\label{eq0:prop2}
\| |\nabla u|_{K^-_{p, \mathfrak C} }\|_{L^1 }\leq \lmti{n}{\infty} \E_n(u) \leq \lmts{n}{\infty} \E_n(u) \leq \| |\nabla u|_{K^+_{p, \mathfrak C} }\|_{L^1 }\leq  \int \theta^+ |\nabla u|^p\,\d \mm.
\end{equation}
By Lemma \ref{lemma4},  we  know $\| |\nabla u|_{K^-_{p, \mathfrak C} }\|_{L^1 }=\| |\nabla u|_{K^-_{p, \mathfrak C} }\|_{L^1 }$. Then \eqref{eq0:coro1} follows from \eqref{eq0:prop2}.

 \end{proof}
 
 \begin{lemma}\label{lemma1}
Given $\epsilon >0$.  As  $\delta \to 0$, we have
 \begin{eqnarray*}
&&\sum_{i=0}^\infty \int_{(1-\epsilon)\delta_{i+1}}^{(1-\epsilon)\delta_i} \int_{S_r^{C} }  \left(\frac{ \delta_i \Big | u_{0,x}\big(D_{\delta_i^{-1}(1-\epsilon)^{-1}}(v)\big )\Big| }{r}\right)^p \tilde \rho_n (r)\,\d\mm^+_{C}(v)\,\d r\\
&=&\big (1+O(\epsilon)\big)\left( \int_{0}^{(1-\epsilon)\delta}  r^{N-1}  \tilde\rho_n (r) \,\d r \right)\left(\int_{S_1^{C} }   \left | u_{0,x}(v)\right| ^p\,\d\mm^+_{C}(v)\right).
 \end{eqnarray*}

 \end{lemma}
 \begin{proof}
 First of all, recall  that $u_{0, x}$ is a uniform limit of rescaling functions, i.e. for any  $v\in C$ and $(x_i)\subset X$ with $x_i \in X$ and $r_i \downarrow 0$,  and any Gromov-Hausdorff approximations $(\varphi_{r_i} )$ (recall that for each $i$, $\varphi_{r_i}$ approximates  $(B_1^C(x_0), \d_C)$ by $(B_{r_i}(x), r^{-1}\d)$ in Gromov-Hausdorff sense)  such that $\varphi_{r_i} (x_i) \to v$,  it holds
 \[
 u_{0, x}(v)=\lmt{i}{\infty} u_{r_i,  x}(x_i).
 \]

By assumption,  for any $\lambda>0$,  $D_{\lambda^{-1}}$ is an isometry from   $(B_\lambda^C(x_0), \lambda^{-1}\d_C)$ to $(B_{1}^C(x_0), \d_C)$,  and $\varphi_{r_i}$ approximates  $(B_{\lambda}^C(x_0), \lambda^{-1} \d_C)$ by $(B_{\lambda r_i}(x), \lambda^{-1}  r_i^{-1}\d)$.  So  $(D_{\lambda^{-1}} \circ \varphi_{r_i} )_i$ is also a sequence of Gromov-Hausdorff approximations. For each $i$, $D_{\lambda^{-1}} \circ \varphi_{r_i}$ approximates $(B_1^C(x_0), \d_C)$ by $(B_{\lambda  r_i}(x), \lambda^{-1}  r_i^{-1}\d)$), and  $D_{\lambda^{-1}}\circ \varphi_{r_i} (x_i) \to D_{\lambda^{-1}}(v)$ as $i\to \infty$. Thus by uniqueness of $ u_{0, x}$, there is an isometry $\varphi$ of $\mathfrak{C}$ such that
  \[
\underbrace{ u_{0, x}\big(\varphi \circ D_{\lambda^{-1}}(v)\big )=\lmt{i}{\infty} u_{\lambda r_i,  x}(x_i)}_{\text{By  Definition}  ~\ref{def:pointConv}}=\underbrace{\lambda^{-1} \lmt{i}{\infty} u_{ r_i,  x}(x_i)= \lambda^{-1} u_{0, x}(v)}_{\text{By  Definition}  ~\ref{def:pointConv}}.
 \]
 
 Therefore, for any $i$, there is an isometry $\varphi$  such that
 \begin{eqnarray*}
 &&  \int_{S_r^{C} }  \Big( \delta_i \left | u_{0,x}\big(D_{\delta_i^{-1}(1-\epsilon)^{-1}}(v)\big )\right| \Big)^p\,\d\mm^+_{C}(v)\\
 &=&\big( (1-\epsilon)^{-1}r \big)^p \int_{S_r^{C} }   \left | u_{0,x}\big(\varphi \circ D_{r^{-1}}(v)\big )\right| ^p\,\d\mm^+_{C}(v)\\
\text{By change of variable}   &=&\big( (1-\epsilon)^{-1}r \big)^p r^{N-1} \int_{S_1^{C} }   \left | u_{0,x}(v)\right| ^p\,\d\mm^+_{C}(v).
 \end{eqnarray*}
 
 
Hence
 \begin{eqnarray*}
 &&\sum_{i=0}^\infty  \int_{(1-\epsilon)\delta_{i+1}}^{(1-\epsilon)\delta_i} \int_{S_r^{C} }  \Big(\frac{ \delta_i \left | u_{0,x}\big(D_{\delta_i^{-1}(1-\epsilon)^{-1}}(v)\big )\right| }{r}\Big)^p \tilde \rho_n (r)\,\d\mm^+_{C}(v)\,\d r\\
 &=&(1-\epsilon)^{-p} \left(\sum_{i=0}^\infty  \int_{(1-\epsilon)\delta_{i+1}}^{(1-\epsilon)\delta_i}  r^{N-1}  \tilde\rho_n (r) \,\d r\right)  \left(\int_{S_1^{C} }   \left | u_{0,x}(v)\right| ^p\,\d\mm^+_{C}(v)\right)\\
  &=&\big (1-p\epsilon+o(\epsilon)\big)\left( \int_{0}^{(1-\epsilon)\delta}  r^{N-1}  \tilde\rho_n (r) \,\d r \right)\left(\int_{S_1^{C} }   \left | u_{0,x}(v)\right| ^p\,\d\mm^+_{C}(v)\right)
 \end{eqnarray*}
 which is the thesis.
 \end{proof}
 
  \begin{lemma}\label{lemma2}
  As $\delta \to 0$, it holds the following uniform estimate (w.r.t. $n$):
  \[
II(\delta, n) =o(1).
  \]
  In other words, for any $\epsilon>0$,  there is $\delta>0$ such that 
\begin{equation}\label{eq1:lemma2}
   II(\delta, n)<\epsilon
\end{equation}
for all $n\in \N$.
 \end{lemma}
 
 \begin{proof}
 
By \eqref{eq0.2:prop2}, as $\delta \to 0$ it holds
\begin{eqnarray*}
&& II(\delta, n)\\
&\lesssim& \sum_{i=0}^{\infty} \delta_i^{p} \int_{ B_{\delta_{i}}(x)\setminus B_{\delta_{i+1}}(x)}A(\delta_i) \frac{\rho_n(x, y)}{\d^{p}(x, y)}  \,\d\mm(y)\\
&\leq &  \alpha(\delta) \sum_{i=0}^{\infty} \delta_i^{p} \int_{ B_{\delta_{i}}(x)\setminus B_{\delta_{i+1}}(x)} \frac{\rho_n(x, y)}{\delta_{i+1}^{p}}  \,\d\mm(y)\\
&=& { 2^p\alpha(\delta)} \int_{ B_{\delta}(x)} \rho_n(x, y)  \,\d\mm(y)\\
&\leq & 2^p\alpha(\delta) c_1
\end{eqnarray*}
 where $c_1=\sup_x \sup_n \|\rho_n(x, \cdot)\|_{L^1}$ and $\alpha(\delta) \to 0$ as $\delta \to 0$.
So for any $\epsilon>0$, there is $\delta_0>0$ such that 
\begin{equation}\label{eq2:lemma2}
  II(\delta,  n)<\epsilon
\end{equation}
for all $\delta\leq \delta_0$ and $n\in \N$.
 \end{proof}
 \bigskip
 
\begin{lemma}\label{lemma3}
Let $x\in X$ be a point with positive density (i.e. $\theta(x)\in (0, +\infty)$, see Lemma \ref{lemma4}).  For any measurable set $E\subset X$ with $\mm\big(E\cap B_r(x)\big)=o(r^N)$, we have a  uniform (w.r.t. $n$) estimate
\[
\sum_{i=0}^{\infty} \int_{\big(B_{\delta_i}(x)\setminus B_{\delta_{i+1}}(x)\big)  \cap E} {\rho_n (x, y)}\,\d\mm(y)=o(1)
\]
 for small $\delta>0$. Similarly, for any $E'\subset C$ with $\mm_C(E'\cap B^C_r(x_0))=o(r^N)$, we have the following  estimate for small $\delta>0$
\[
 \sum_{i=0}^{\infty} \int_{\big(B^C_{\delta_i}(x_0)\setminus B^C_{\delta_{i+1}}(x_0)\big)  \cap E'} {\rho_n (x_0, v)}\,\d\mm_C(v)=o(1).
\]

In particular,  it holds the following uniform estimate (w.r.t. $n$):
\[
 III(\delta,  n)=o(1),
\]
i.e. for any $\epsilon>0$, there is $\delta_0>0$ such that 
\[
 III(\delta,   n) <\epsilon,~~~~ \forall n\in \N,~\forall \delta<\delta_0.
\]

 \end{lemma}
 
 \begin{proof}

We  will only  prove the estimate concerning
\[
\sum_{i=0}^{\infty} \int_{\big(B_{\delta_i}(x)\setminus B_{\delta_{i+1}}(x)\big)  \cap \NN} {\rho_n (x, y)}\,\d\mm(y).
\]
The second estimate can be proved in the same way.

  By  assumption  there is  an increasing  function $\eta$ satisfying $\lmt{r}{0}\eta(r)=0$, such that  $\mm(E \cap B_r)  \leq \eta(r) r^N$. Combining with the monotonicity of $\tilde \rho$, we have
 \begin{eqnarray*}
&&  \sum_{i=0}^{\infty}   \int_{\big(B_{\delta_i}(x)\setminus B_{\delta_{i+1}}(x)\big)  \cap E } {\rho_n (x, y)}\,\d\mm(y) \\
&\leq &  \sum_{i=0}^{\infty}  \eta(\delta_i) \delta_i^N {\tilde \rho_n (\delta_{i+1})}\\
&=&  \sum_{i=0}^{\infty}  \eta(\delta_i)\frac{ \delta_{i}^N}{\delta_{i+1}^N -\delta_{i+2}^N } \Big( (\delta_{i+1}^N -\delta_{i+2}^N ){\tilde \rho_n (\delta_{i+1})}\Big)\\
&\lesssim&   \eta(\delta) \sum_{i=0}^{\infty}   \int_{\big(B_{\delta_{i+1}}(x)\setminus B_{\delta_{i+2}}(x)\big) }  {\rho_n (x, y)}\,\d\mm(y)\\
&\leq& \eta(\delta) c_1
 \end{eqnarray*}
which is the thesis.
 \end{proof}

\bigskip

\begin{lemma}\label{lemma5}
Denote $${\mathcal M}^b_{i}(\delta):=\phi_{\delta_i}\Big( B_{\delta_i}(x)\setminus  \big( B_{\delta_{i+1}}(x) \cup \NN\big) \Big)\setminus  \big(B^C_1(x)\setminus  B^C_{\frac 12} (x) \big) $$
and
$${\mathcal M}^c_{i}(\delta):=\big(B^C_1(x)\setminus  B^C_{\frac 12}  (x) \big)  \setminus \phi_{\delta_i}\Big( B_{\delta_i}(x)\setminus  \big( B_{\delta_{i+1}}(x) \cup \NN\big) \Big).$$
Then for $\mm$-a.e. $x$, as $\delta \downarrow 0$, we have
\[
\mm_C({\mathcal M}^b_{i})+\mm_C({\mathcal M}^c_{i})= o(1).
\]
\end{lemma}
\begin{proof}
Firstly,  for $\delta$ small enough we know
\begin{equation}\label{eq1:sv}
\phi_{\delta_i}\Big( B_{\delta_i}(x)\setminus  \big( B_{\delta_{i+1}}(x) \cup \NN\big) \Big)\subset \big(B^C_{1+o(1)}(x)\setminus  B^C_{{(1+o(1))}\frac{1}2} (x) \big)
\end{equation}
So
\begin{eqnarray*}
{\mathcal M}^b_{i}(\delta) &\subset&   \big(B^C_{1+o(1)}(x)\setminus  B^C_{{(1+o(1))}\frac{1}2} (x) \big) \setminus   \big(B^C_1(x)\setminus  B^C_{\frac 12} (x) \big)\\
&\subset&   \big(B^C_{1+o(1)}(x)\Delta  B^C_1(x) \big) \cup   \big(B^C_{{(1+o(1))}\frac{1}2} (x)\Delta  B^C_{\frac 12} (x) \big)
\end{eqnarray*}
By homogeneity of the tangent space $\mathfrak C$, we know $$\mm_C({\mathcal M}^b_{i})=\mm_C \big(B^C_1(x)\setminus  B^C_{\frac{1}2} (x) \big) o(1).$$

Secondly, to prove $\mm_C({\mathcal M}^c_{i})=o(1)$,  by \eqref{eq1:sv} we just need to show that
\[
\mm_C\left (\phi_{\delta_i}\Big( B_{\delta_i}(x)\setminus  \big( B_{\delta_{i+1}}(x) \cup \NN\big) \Big) \right)\geq \big(1+o(1)\big)\mm_C \big(B^C_1(x)\setminus  B^C_{\frac{1}2} (x) \big).
\]

By Assumptions  \ref{assumption1}-3) and 4) we have
\begin{eqnarray*}
&&\mm\Big(B_{\delta_i}(x)\setminus  \big( B_{\delta_{i+1}}(x) \cup \NN\big)\Big)\\
&=&\big(1+o(1)\big) \mm\Big(B_{\delta_i}(x)\Big) \mm_C\left (\phi_{\delta_i}\Big( B_{\delta_i}(x)\setminus  \big( B_{\delta_{i+1}}(x) \cup \NN\big) \Big) \right)\\
&\leq&\big(1+o(1)\big) \theta^+(x)\delta_i^N \mm_C\left (\phi_{\delta_i}\Big( B_{\delta_i}(x)\setminus  \big( B_{\delta_{i+1}}(x) \cup \NN\big) \Big) \right)
\end{eqnarray*}
and
\begin{eqnarray*}
&&\mm\Big(B_{\delta_i}(x)\setminus  \big( B_{\delta_{i+1}}(x) \cup \NN\big)\Big)+\eta(\delta)\delta_i^N\\
&\geq&\mm\Big(B_{\delta_i}(x)\setminus  B_{\delta_{i+1}}(x) \Big)\\
&=&  \Big(\mm\big(B_{\delta_i}(x)\big)-\mm\big(B_{\delta_{i+1}}(x)\big) \Big )\\
&\geq & \big(1+o(1)\big) \Big(\theta^-(x)\delta_i^N -\theta^+(x)\delta_{i+1}^N  \Big ).
\end{eqnarray*}

Combining with Lemma \ref{lemma4} we obtain
\[
 \mm_C\left (\phi_{\delta_i}\Big( B_{\delta_i}(x)\setminus  \big( B_{\delta_{i+1}}(x) \cup \NN\big) \Big) \right)\geq  \big(1+o(1)\big) \Big(1-\big(\frac{1}2\big)^N  \Big )+o(1)
\]
for $\mm$-a.e. $x$, which is the thesis.

\end{proof}
\bigskip

\begin{lemma}\label{lemma6}
 Given $\epsilon>0$.  For $\delta$ small enough, it holds the following uniform estimate (w.r.t. $n$):
  \[
\lmt{m}{\infty}    \mathop{\sum}_{i=1}^{m}\big(I_b(i, \delta, m, n)+I_c(i, \delta, m, n)\big)<\epsilon.
\]
\end{lemma}
\begin{proof}

By  Lemma \ref{lemma5}  we know ${\mathcal M}^b_{i}+{\mathcal M}^c_{i} \subset B^C_{\frac 32}(x)\setminus  B^C_{\frac{1}2} (x)$ and  $\mm_C({\mathcal M}^b_{i})+\mm_C({\mathcal M}^c_{i})= o(1)$, so $\mm_C \big(D_{\delta_i} ({\mathcal M}^b_{i}+{\mathcal M}^c_{i})\big)=o(\delta^N)$.
Then the lemma can be proved in a similar way as Lemma \ref{lemma3},  with the help of a change of variable.

\end{proof}

\bigskip

By  Proposition \ref{prop2}  and the  density of Lipschitz functions in $W^{1, p}$ (cf. Proposition \ref{prop:density}), we immiediately obtain the following lemma concerning the well-posedness of $\| \nabla u \|^p_{K_{p, \mathfrak C} }$ for any  $u \in W^{1, p}(X, \d, \theta \mm)$, where $\theta$ is the density function given in Lemma \ref{lemma4}.
\begin{lemma}\label{lemma:norm}
The semi-norm $\| \nabla \cdot \|^p_{K_{p, \mathfrak C} }$ is well-defined on $W^{1, p}(X, \d, \theta \mm)$, i.e. for any   $u\in W^{1, p}(X, \d, \theta \mm)$, $ \| \nabla u \|^p_{K_{p, \mathfrak C} }$ can be uniquely defined by
\[
\| \nabla u \|^p_{K_{p, \mathfrak C} }:=\lmt{k}{\infty} \| \nabla u_k \|^p_{K_{p, \mathfrak C} }
\]
for any sequence $(u_k)$ of Lipschitz functions converging to $u$ in $W^{1,  p}(X, \d, \theta \mm)$. In other words,  the value of $\lmt{k}{\infty} \| \nabla u_k \|^p_{K_{p, \mathfrak C} }$ is independent of the choice of $(u_k)$.
\end{lemma}
\bigskip
 
 \begin{theorem}\label{th1}
Let $u\in  W^{1, p}(X,\d, \theta\mm)$. Assume that  $\theta^{-1} \in L^\infty$. Then we have the same  generalized Bourgain-Brezis-Mironescu's  formula  \eqref{eq0:coro1} as 
in Proposition \ref{prop2}.
 \end{theorem}
 
 \begin{proof}
 Let $(u_k)$ be a sequence of   Lipschitz functions with compact support such that $u_k \to u$ strongly in $W^{1,p}{(X,\d,  \theta\mm)}$. For any $\epsilon\in (0, 1)$, there is $k_0\in \N$ such that 
 \begin{equation}\label{eq:th1}
 \|u-u_{k_0}\|_{W^{1,p}(X,\d, \theta\mm)}<\epsilon.
\end{equation}
By Lemma \ref{lemma:norm} we can also assume
\begin{equation}\label{eq2:th1}
\left| \| \nabla u_{k_0}\|_{K_{p, \mathfrak C} }- \| \nabla u\|_{K_{p, \mathfrak C} }\right| <\epsilon
\end{equation}
 
By Proposition \ref{prop:density2}  (Hajlasz-Sobolev space), there is $g\in L^p(X,\theta\mm)$ such that 
\[
\left|\big (u_{k_0}(x)-u(x)\big)-\big (u_{k_0}(y)-u(y)\big)\right| \leq \d(x, y) \big(g(x)+g(y) \big)
\]
and there is a constant $C>0$ which is independent of $u$ and $u_{k_0}$,  such that
\[
 \|g\|_{L^p(X,\theta\mm)} \leq C \|\D (u-u_{k_0})\|_{L^p}(X,\theta\mm).
\]
Hence
 \begin{eqnarray*}
&& \left|\E^{\frac1p}_n(u_{k_0})- \E^{\frac1p}_n(u)\right| ^p
\leq  \E_n(u_{k_0}-u)\\
&\leq & \int_X \int_X \big(g(x)+g(y)\big)^p\rho_n(x, y)\,\d \mm(x)\d \mm(y)\\
&\leq&2^p c_1 \| \theta^{-1}\|_{L^\infty}\|g\|^p_{L^p(X,\theta\mm)} \leq 2^p c_1C^p\| \theta^{-1}\|_{L^\infty} \|\D (u-u_{k_0})\|^p_{L^p(X,\theta\mm)}\\
&\leq& 2^p c_1C^p\| \theta^{-1}\|_{L^\infty} \|u-u_{k_0}\|^p_{W^{1,p}(X,\d, \theta\mm)}\\
&<& 2^p c_1C^p\| \theta^{-1}\|_{L^\infty}  \epsilon^p.
 \end{eqnarray*}

  By  Proposition \ref{prop1} and \eqref{eq:th1}, there is $n_0\in \N$ such that for any $n >n_0$, it holds
\begin{equation}\label{eq1:th1}
\left|\E^{\frac1p}_n(u_{k_0})- \| \nabla u_{k_0}\|_{K_{p, \mathfrak C} }\right| <\epsilon.
\end{equation}

Combining  the estimates above, we obtain
\begin{eqnarray*}
&&\left|\E^{\frac1p}_n(u)- \| \nabla u\|_{K_{p, \mathfrak C} }\right|\\
&\leq&\left|\E^\frac1p_n(u)-\E^\frac1p_n(u_{k_0})\right|+ \left|\E^{\frac1p}_n(u_{k_0})- \| \nabla u_{k_0}\|_{K_{p, \mathfrak C} }\right| +\left| \| \nabla u_{k_0}\|_{K_{p, \mathfrak C} }- \| \nabla u\|_{K_{p, \mathfrak C} }\right|\\
&\leq&2^p c_1C^p\| \theta^{-1}\|_{L^\infty} \epsilon^p+2\epsilon
\end{eqnarray*}
 for any $n >n_0$, which is the thesis.


 \end{proof}

 

 \section{Examples and Applications}
 The first proposition extends a result of M.Ludwig \cite{Ludwig14} concerning finite dimensional Banach spaces with general mollifiers.  We remark that the proof in \cite{Ludwig14} is based on Blaschke–Petkantschin formula, which  works only  for the mollifiers $\rho_n=\frac{1/n}{\|x-y\|^{N-p/n}}$.
 
 \begin{proposition}[Anisotropic spaces]\label{prop:banach}
Let $\mathfrak C=(\R^N, \|\cdot\|, \mathcal L^N)$ be a  $N$-dimensional Banach space equipped with the Lebesgue measure $\mathcal L^N$ for some  $N\in \N$.  Then
\begin{equation}\label{eq1:coro1}
 \mathop{\lim}_{n \to \infty} \int_{\R^N} \int_{\R^N} \frac{|f(x)-f(y)|^p}{\|x- y\|^p} \rho_n(x, y)\,\d x\, \d y= \| \nabla f\|^p_{K_{p, \mathfrak C} }
\end{equation}
where
\[
\| \nabla f\|^p_{K_{p, \mathfrak C} }= C_0\int_{\R^N} \int_{S_1^{\mathfrak C} } |\nabla f \cdot v |^p \,\d\mathcal H_{\|\cdot\|}^{N-1}(v)\,\d x
\]
where $B^{\mathfrak C}_1$ is the unit ball in $\mathfrak C$ centred at $0$ and $S_1^{\mathfrak C} $ is its boundary, $\mathcal H_{\|\cdot\|}^{N-1}$ is the boundary measure (Minkowski content)  w.r.t. $\|\cdot \|$ and $\mathcal L^N$, and the constant $C_0$ is given by
\[
 C_0=\lmt{n}{\infty}  \int_0^\delta  r^{N-1} \tilde \rho_n(r)\,\d r.
\]
\end{proposition}

\begin{proof}
 Obviously $\mathfrak C$ is a Cheeger differentiable space,  and its tangent spaces are isometric to the normalized space $\overline{ \mathfrak{C}}:=(\R^N, \|\cdot\|, (\mathcal L^N(B^{\mathfrak C}_1))^{-1}\mathcal L^N)$ .  Let $x\in \R^N$ be  a differentiable point of a Lipschitz function $f$ with respect to both $\| \cdot \|$ and  the Euclidean norm $|\cdot|$. By \cite[Theorem 10.2]{C-D} and the Rademacher's theorem on Euclidean spaces we know the union of  such points have full measure. Let $(\varphi_r)_{r>0}$ be a family of dilations with fix point $x$, i.e. 
 $\varphi_r(y)=x+r(y-x)$.
 It can be seen that   $(\varphi_r)_{r>0}$  is a family of good Gromov-Hausdorff approximation maps from $\mathfrak C$ to the corresponding rescaled spaces.

Assume (up to taking a subsequence) $(f_{r, x})_r$  converges uniformly to a  limit $f_{0, x}$ with respect to $\| \cdot \|$ as $r\to 0$.
By \cite[Theorem 10.2]{C-D},  $f_{0,x}$ is a generalized linear function on $\mathfrak{C}$. Meanwhile, since the norm  $\|\cdot\|$ and the Euclidean norm are equivalent,  $(f_{r, x})_r$  also converges uniformly to $f_{0, x}$  with respect to the Euclidean norm.  By Rademacher's theorem on Euclidean spaces, we know that $f_{0,x}$  can be written in the unique way, as a  linear function $f_{0,x}(v)=\nabla f  \cdot v$ for any $v\in \R^N$.  Then the assertion follows from Theorem \ref{th1}.

\end{proof}

\begin{remark}
In \cite{Ludwig14}, the author studies  $\mathop{\lim}_{s \to 1^-} (1-s)\int_{\R^N} \int_{\R^N} \frac{|f(x)-f(y)|^p}{\|x- y\|^{N+sp}} \,\d x\, \d y$. Applying our proposition  to this special case, we have
\begin{eqnarray*}
&&\mathop{\lim}_{s \to 1^-} (1-s)\int_{\R^N} \int_{\R^N} \frac{|f(x)-f(y)|^p}{\|x- y\|^{N+sp}} \,\d x\, \d y\\
&=&C_0 \int_{\R^N} \int_{S_1^{\mathfrak C} } |\nabla f \cdot v |^p \,\d \mathcal H_{\|\cdot\|}^{N-1}(v)\,\d x
\end{eqnarray*}
where by definition the constant $C_0$ is
\[
\lmt{s}{ 1^-}  (1-s)\int_0^\delta r^{N-1} r^{-N-sp+p}\,\d r=\frac 1p.
\]
Note that 
\begin{eqnarray*}
&&  \int_{B^{\mathfrak C}_1} |\nabla f \cdot v |^p \,\d \mathcal H_{\|\cdot\|}^{N-1}(v)\\
&=& \int_0^1 \int_{S_r^{\mathfrak C} } |\nabla f \cdot v |^p \,\d \mathcal H_{\|\cdot\|}^{N-1}(v)\,\d r\\
\text{By change of variable} &=& \int_0^1 \int_{S_1^{\mathfrak C} }  |\nabla f \cdot  r v |^p \,\d \mathcal H_{\|\cdot\|}^{N-1}(r v)\,\d r\\
&=&   \int_0^1 r^{p+N-1} \Big( \int_{S_1^{\mathfrak C} }   |\nabla f \cdot  v |^p \,\d \mathcal H_{\|\cdot\|}^{N-1}(v)\Big)\,\d r\\
&=& \frac 1{p+N}   \int_{S_1^{\mathfrak C} }   |\nabla f \cdot  v |^p \,\d \mathcal H_{\|\cdot\|}^{N-1}(v).
\end{eqnarray*}

In conclusion, we have
\begin{eqnarray*}
&&\mathop{\lim}_{s \to 1^-} (1-s)\int_{\R^N} \int_{\R^N} \frac{|f(x)-f(y)|^p}{\|x- y\|^{N+sp}} \,\d x\, \d y\\
&=&\frac {p+N}p \int_{\R^N} \int_{B^{\mathfrak C}_1} |\nabla f \cdot v |^p \,\d \mathcal H_{\|\cdot\|}^{N-1}(v)\,\d x
\end{eqnarray*}
where the unit ball $B^{\mathfrak C}_1 $ is a convex body associated with the norm $\| \cdot \|$.

\end{remark}

\bigskip

As a corollary, we get  Bourgain-Brezis-Mironescu's formula on $\R^n$.

 \begin{proposition}[Euclidean spaces]\label{prop:eu}
Let $\mathfrak C=(\R^N, |\cdot|, \mathcal L^N)$ be the $N$-dimensional Euclidean space equipped with the Euclidean distance and the Lebesgue measure, and $(\rho_n)_n$ be a family of mollifiers satisfying Assumption \ref{assumption2}. Then
\begin{equation}\label{eq:prop:eu}
 \mathop{\lim}_{n \to \infty} \int_{\R^N} \int_{\R^N} \frac{|f(x)-f(y)|^p}{|x- y|^p} \rho_n(x, y)\,\d x\, \d y=K_{p, N}  \| \nabla f\|^p_{L^p}
\end{equation}
where
\[
K_{p, N}:= \mathcal L^N(B^N_1)\fint_{S_1^N } |w\cdot v |^p \,\d\mathcal H^{N-1}(v)
\]
is a constant independent of the choice of $w$ in an $N$-dimensional unit sphere $S_1^N$, and  $ \mathcal L^N(B^N_1)=\frac{\pi^{\frac N 2}}{\Gamma(\frac N2+1)}$ is the volume of an $N$-dimensional unit ball.
\end{proposition}
\begin{proof}
It can be seen from Rademacher's theorem that the limit of rescaling functions is unique, which is a linear function.
Notice also that $$ \fint_{S_1^{N} } |\nabla f \cdot v |^p \,\d\mathcal H^{N-1}(v)=|\nabla f |^p \fint_{S_1^{N} }  \left|\frac{\nabla f}{|\nabla f |} \cdot v \right|^p \,\d\mathcal H^{N-1}(v).$$

By isotropicity of the Euclidean space, we know that  $K_{p, N}$ is a constant, which is independent of $w\in S_1^N$.  Therefore  $$ \fint_{S_1^{N} }  \left|\frac{\nabla f}{|\nabla f |} \cdot v \right|^p \,\d\mathcal H^{N-1}(v)=\fint_{S_1^N } |w\cdot v |^p \,\d\mathcal H^{N-1}(v)~~~~\forall w\in S_1^N.$$
Then the assertion follows from Proposition \ref{prop:banach}.
\end{proof}

\bigskip

 \begin{proposition}[Weighted Riemannian manifolds, cf. \cite{BBM-RM}]\label{prop:weight}
Let $(M, \d_\g, \mm)$ be a $N$-dimensional weighted Riemannian manifold  with $\mm=\theta \vol$,  $\theta \in C(M)$ and $\theta^{-1}\in L^\infty$, and $(\rho_n)_n$ be a family of mollifiers satisfying Assumption \ref{assumption2}. Then 
\begin{equation}\label{eq1:coro2}
 \mathop{\lim}_{n \to \infty} \int_M \int_M \frac{|f(x)-f(y)|^p}{\d^{p}(x, y)} \rho_n(x, y)\,\d\mm(x) \d\vol(y)= {K_{p, N} }  \| \nabla f\|_{L^p(M, \mm)}^p
\end{equation}
where
\[
K_{p, N}:= \mathcal L^N(B^N_1)\fint_{S_1^N } |w\cdot v |^p \,\d\mathcal H^{N-1}(v)
\]
is the same constant as the constant appeared in \eqref{eq:prop:eu}.
\end{proposition}
\begin{proof}
We can choose the Gromov-Hausdorff approximations $(\varphi_\delta)$ induced by the exponential maps. Then the assertion can be proved using Theorem \ref{th1} and the proof 
of Proposition \ref{prop:eu}.
\end{proof}

Next we will apply our results to  equi-regular sub-Riemannian manifolds.   It is known that an equi-regular  sub-Riemnnian manifold  has a unique tangent cone, which is isometric to a Carnot group.  On a Carnot group, it is also known that Haar, Lebesgue and the top-dimensional Hausdorff measures are left-invariant and proportional. Without loss of generality, we will  neglect the multiplicative constants.

 \begin{proposition}[Sub-Riemannian manifolds]\label{prop:sr}
 Let $\ms$ be an $m$-dimensional  equi-regular sub-Riemannian manifold with homogeneous (Hausdorff) dimension $N \geq m$, equipped with the Carnot-Carath\'eodory metric $\d$ and the associated Hausdorff measure $\mm=\mathcal H^N_\d$. Let $(\rho_n)_n$ be a family of mollifiers satisfying Assumption \ref{assumption2}. 
We have
\begin{equation}\label{eq1:coro3}
 \mathop{\lim}_{n \to \infty} \int_X \int_X \frac{|f(x)-f(y)|^p}{\d^{p}(x, y)} \rho_n(x, y)\,\d\mm(x) \d \mm(y)=
\| \nabla f\|^p_{K_{p, \mathfrak C} }
\end{equation}
where
\[
\| \nabla f\|^p_{K_{p, \mathfrak C} }=\mathcal L^m(B^{\mathfrak C}_1) \int_{X} \fint_{S_1^{\mathfrak C} } |\nabla f \cdot v |^p \,\d\mathcal H^{N-1}_{\d_{CC}}(v)\,\d \mm
\]
where $B^{\mathfrak C}_1$ is the unit ball in the tangent cone $\mathfrak C=(\R^m, \d_{CC}, \mathcal L^m=\mathcal H^{N}_{\d_{CC}})$ centred at $0$,   $S_1^{\mathfrak C} $ denotes  its boundary and $\mathcal H^{N-1}_{\d_{CC}}$ is the boundary measure.
\end{proposition}
\begin{proof}
We just need to check the items in Assumption \ref{assumption1} one by one.

1) By Pansu's theorem \cite[Th\'eor\`em 2]{Pansu89} concerning Rademacher's theorem  on sub-Riemannian manifolds, we know Lipschitz functions are almost everywhere differentiable and the limit of rescaling functions is unique. 

2) It was proved by Mitchell in \cite[Theorem 1]{Michell85} (see also \cite{Pansu89}) that the tangent cone of a sub-Riemannian manifold equipped  with a equi-regular (or called generic) distribution  $\{X_i\}_{i=1,...,k}$, $k\leq n$, is isometric to
a nilpotent Lie group (Carnot group) with a left-invariant Carnot-Carath\'eodory metric $\d_{CC}$.  By  \cite{Ghezzi15} we also  know the tangent cone, as a metric measure space,  is also  a Carnot group $(C, \d_{CC}, \mathcal L^n)$ equipped with its Haar (Hausdorff, Lebesgue) measure.

3) and 4) Similar to Riemannian cases, there are almost isometries  $(\phi_\delta)$ induced by exponential maps on sub-Riemannian manifolds. 
More precisely,  there exist positive constants $c$  and $r$,  such that  (see \cite[Theorem 6.4]{Bellaiche97})
\begin{equation}\label{eq1:prop3}
-c \d(x, y)\Big(\delta \d_{CC}\big(x_0, \phi_\delta(y)\big) \Big)^{\frac 1r} <\d(x, y)-\delta \d_{CC}\big(x_0, \phi_\delta(y)\big) <c\d(x, y) \Big(\delta \d_{CC}\big(x_0, \phi_\delta(y)\big) \Big)^{\frac 1r}
\end{equation}
for some $c>0$.
Denote $\delta \d_{ CC}\big(x_0, \phi_\delta(y)\big)$ by $|w|$. By \eqref{eq1:prop2}  and an iteration  argument we have 
\begin{eqnarray*}
\d(x, y) &<&  |w|+c\d(x, y)|w|^{\frac 1r} <|w|+c\big(|w|+c\d(x, y)|w|^{\frac 1r}\big) |w|^{\frac 1r}<...\\
&<& |w|(1+O(|w|)).
\end{eqnarray*}
Similarly, we can prove
\[
\d(x, y)> |w|(1+O(|w|)).
\]

Thus
\[
\frac{\d(x, y)}{\delta \d_{ CC}\big(x_0, \phi_\delta(y)\big)}=1+O(\delta)
\]
which is the thesis.

5) On Carnot group, there is a canonical dilation map $\D_r: C \to C$ such that $$\d_{CC}\big(D_r(x), D_r(y) \big)=r\d_{C}(x, y),$$ and $(D_r)_\sharp \mathcal L^n=r^{-N} \mathcal L^n$ where $N$ is the Hausdorff dimension w.r.t. $\d_{CC}$. It is known that $N$ can also be computed by  $$N=\sum_i i({\rm dim}(V_i)-{\rm dim}(V_{i-1}))$$ where $V_i(x)$ is the subspace of $T_xM$ spanned by all commutators of the distributions $X_j$ of order $\leq$ i.
\end{proof}

\bigskip

At last, we will  prove Bourgain-Brezis-Mironescu's formula for non-collapsed RCD$(K, N)$ metric measure spaces introduced  by De Philippis and Gigli in  \cite{DPG-N}. The synthetic theory of metric measure spaces with lower Ricci curvature bounds (and upper dimension bounds), initiated by Lott-Villani \cite{Lott-Villani09} and Sturm \cite{S-O1, S-O2},  has great developments in recent years.  We refer the reader to the ICM proceeding \cite{AmbrosioICM} by Ambrosio for an overview of this topic  and  bibliography. 

 \begin{proposition}[Non-collapsed RCD$(K, N)$ spaces]\label{prop:nc}
Let $\ms=(X, \d, \mathcal H^N)$ be a  non-collapsed $\rcdkn$ metric measure space.     Then 
\begin{equation}\label{eq1:coro4}
 \mathop{\lim}_{n \to \infty} \int_X \int_X \frac{|f(x)-f(y)|^p}{\d^{p}(x, y)} \rho_n(x, y)\,\d\mm(x) \d\mm (y)= {K_{p, N} }  \| \nabla f\|_{L^p\ms }^p
\end{equation}

\end{proposition}
\begin{proof}
 By  a structural theorem of Mondino-Naber \cite{MN-S} we know  $(X,\d,\mm)$ is rectifiable.   For any $\delta>0$,  there is a family $\{U_\delta^i, \varphi_\delta^i)\}_{i\in \N}$ of Borel charts  $U_\delta^i\subset X$ such that $X=\cup_i U_\delta^i$ up to a $\mm$-negligible set, and $\varphi_\delta^i: X \to \R^{N}$ is $(1+\delta)$-bi-Lipschitz, i.e. for any $x, y \in X$,
\[
(1+\delta)^{-1} \d(x, y) \leq \left |\varphi_\delta^i(x)-\varphi_\delta^i(y) \right | \leq (1+\delta) \d(x, y).
\] 
Define $R_\delta \subset \cup_{i\in \N} U_\delta^i$ as the union of   density-1 points $x\in U_\delta^i$ for some $i\in \N$, i.e. $x\in R_\delta $ if and only if there is $i\in \N$ such that $x\in U_\delta^i$ and
\begin{equation}\label{eq2:coro4}
\lmt{r}{0}\frac{\mm(B_r(x)\cap U_\delta^i)}{\mm(B_r(x))}=1.
\end{equation}
It can be seen that $R_\delta$ has full measure in $X$. 

Let  $R$ be the set of regular points (i.e. the union of points where the tangent space is unique and isometric to $\R^N$).  Denote $R_\infty:=\cap_{k\in \N} R_{\frac 1 k} \cap R$.  By \cite{MN-S} and the discussions above we know $R_\infty$ has full measure in $X$.

 Fix $x\in R_\infty$. By definition, for any $k$ there is $i_k\in \N$ such that   $\varphi^{i_k}_{\frac 1k}$ is $(1+\frac 1k)$-bi-Lipschitz on $U_{\frac 1k}^{i_k} \ni x$. Without loss of generality, we can assume that $\varphi^{i_k}_{\frac 1k}(x)=0$.

By \eqref{eq2:coro4},  there is  $r_k \in (0, 1)$ such that for any $r\leq r_k$ it holds
\[
\frac{\mm(B_r(x)\cap U_{\frac 1k}^{i_k})}{\mm(B_r(x))}>1-\frac 1 {2^k}, 
\]
and we denote $\mathcal N_k:= B_{r_k}(x)\setminus U_{\frac 1k}^{i_k}$.
By induction,  we can find $U_{\frac1{k+1}}^{i_{k+1}} \ni x$ and $r_{k+1}<\frac{r_k}{2^k}$ such that
\[
\frac{\mm(B_{r_{k+1}}(x)\cap U_{\frac 1{k+1}}^{i_{k+1}})}{\mm(B_{r_{k+1}}(x))}>1-\frac 1 {2^{k+1}}.
\]

Denote $\mathcal N:=\cup_k  \mathcal N_k=\cup_k \Big(B_{r_k}(x)\setminus U_{\frac 1k}^{i_k}\Big)$. For any $\delta>0$, there is $k_0=k_0(\delta)$ such that $r_{k_0+1} <\delta \leq r_{k_0}$, then
\begin{eqnarray*}
\mm\big (\mathcal N\cap B_\delta(x)\big)
&\leq &\sum_{k\geq k_0} \mm\big (B_{r_k}(x)\setminus U_{\frac 1k}^{i_k}\big)
<  \sum_{k\geq k_0} \frac 1 {2^{k}} \mm\big (B_{r_k}(x)\big)\\
&\lesssim& \sum_{k\geq k_0} \frac 1 {2^{k}} r_k^N
\lesssim \sum_{k\geq k_0} \frac 1 {2^{k}} \delta^N.
\end{eqnarray*}
Note $k_0 \to +\infty$ as $\delta\to 0$, we have
\[
\lmt{\delta}{0} \frac{ \mm\big (\mathcal N\cap B_\delta(x)\big)}{\delta^N}\lesssim \lmt{k_0}{+\infty} \sum_{k\geq k_0} \frac 1 {2^{k}}=0
\]
which fulfils \eqref{eq0:asmp}.

In addition, for any  $\delta>0$ with  $r_{k+1} <\delta \leq r_{k}$,  we define 
$\varphi_\delta:=\varphi^{i_{k}}_{\frac 1{k}}$, which is a desirable  family of good Gromov-Hausdorff approximation satisfying Assumption \ref{assumption1}.
\end{proof}
\begin{remark}
In this proposition,  a  non-collapsed $\rcdkn$ metric measure space is  defined as $(X, \d, \mathcal H^N)$. This  is slightly different from the  original definition of non-collapsed $\rcdkn$ condition in \cite{DPG-N}. However,   by a recent result of Brena-Gigli-Honda-Zhu  \cite{BGHZ2021} (see also \cite{HondaGT}) this definition is equivalent to De Philippis-Gigli's.  

For a general $\rcdkn$ space $\ms$, if the density of $\mm$ w.r.t. the Hausdorff measure has a strictly positive lower bound, we can  prove a non-smooth version  of Proposition \ref{prop:weight} by a similar proof as Proposition \ref{prop:nc}.
\end{remark}


\def\cprime{$'$}

\end{document}